\documentclass{amsart}
\usepackage{amssymb}
\usepackage{amsbsy}
\usepackage{slashed}
\usepackage[fontsize=11pt]{scrextend}
\usepackage[utf8]{inputenc}
\usepackage[margin=1in]{geometry}
\usepackage{enumerate}
\usepackage{xcolor}
\usepackage{tikz-cd}
\usepackage{upgreek}
\usepackage{graphicx}
\usepackage{enumitem}
\usepackage{hyperref}
\makeatletter
\DeclareRobustCommand\widecheck[1]{{\mathpalette\@widecheck{#1}}}
\def\@widecheck#1#2{%
    \setbox\z@\hbox{\m@th$#1#2$}%
    \setbox\tw@\hbox{\m@th$#1%
       \widehat{%
          \vrule\@width\z@\@height\ht\z@
          \vrule\@height\z@\@width\wd\z@}$}%
    \dp\tw@-\ht\z@
    \@tempdima\ht\z@ \advance\@tempdima2\ht\tw@ \divide\@tempdima\thr@@
    \setbox\tw@\hbox{%
       \raise\@tempdima\hbox{\scalebox{1}[-1]{\lower\@tempdima\box
\tw@}}}%
    {\ooalign{\box\tw@ \cr \box\z@}}}
\makeatother

\theoremstyle{definition}
\newtheorem{thm}{Theorem}[section]
\newtheorem{defn}[thm]{Definition}
\newtheorem{rem}[thm]{Remark}

\newtheorem{lem}[thm]{Lemma}

\newtheorem{exmp}[thm]{Example}

\newtheorem{cor}[thm]{Corollary}
\newtheorem{prop}[thm]{Proposition}
\newtheorem{defn-prop}[thm]{Definition-Proposition}
\newtheorem{prop-defn}[thm]{Proposition-Definition}
\newtheorem{thm*}{Theorem*}[section]
\newtheorem{prop*}[thm*]{Proposition}

\DeclareMathOperator{\grad}{grad}
\DeclareMathOperator{\red}{red}

\DeclareMathOperator{\reals}{\mathbb R}

\DeclareMathOperator{\C}{\mathbb{C}}

\DeclareMathOperator{\HMR}{\textit{HMR}}
\DeclareMathOperator{\HM}{\textit{HM}}

\DeclareMathOperator{\ind}{\mathsf{ind}}

\newcommand{\del}{\ensuremath{\partial}}
\newcommand{\delbar}{\ensuremath{{\bar{\partial}}}}

\newcommand{\pertL}{{\slashed{\mathcal L}}}

\newcommand{\gr}{\ensuremath{{\sf{gr}}}}
\newcommand{\Ball}{\ensuremath{{\sf{Ball}}}}

\title{Real Monopole Floer Homology and Skein Exact Triangles}
\date{\today}
\author{Jiakai Li}

\usepackage{hyperref}
\usepackage{standalone}
\usepackage{import}

\address{Dept. of Math., 
Harvard Univ., 
Cambridge MA, 
United States 02138}

\email{jiakaili@math.harvard.edu}

\begin{document}
\begin{abstract}
 	We prove an unoriented skein extract triangle for the real monopole Floer homology and introduce a Fr\o yshov-type invariant.
\end{abstract}
\maketitle
%\tableofcontents

\section{Introduction}
In \cite{ljk2022}, the author constructed \emph{real monopole Floer homologies} $\HMR^{\circ}(K)$ for links $K$ in the 3-sphere.
This paper relates the real monopole Floer homologies of an \emph{unoriented skein triple} of links $(K_0, K_1, K_2)$.
That is,
$K_i$'s differ in a small $3$-ball as in the following figure.
\begin{figure}[hb]
  \includegraphics[height=1.1in]{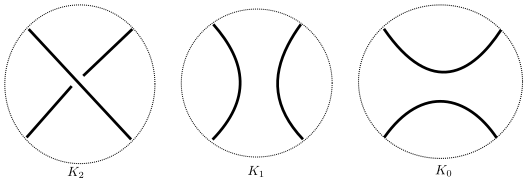}
\caption{Local pictures in an unoriented skein triangle.}  
\end{figure}

The construction of $\HMR^{\circ}$ is based on Kroheimer-Mrowka's monopole Floer homology \cite{KMbook2007}.
The idea is to consider the double branched cover along $K$ and invariant subspaces of the ordinary Seiberg-Witten configuration space under an anti-linear involution.
Such an involution is defined using the covering involution $\upiota$ of the double branched cover and an anti-linear lift $\uptau$ of $\upiota$ on the spinor bundle.
There are three flavours of the real monopole Floer homology 
$$ 
\widehat{\HMR}_*, \widecheck{\HMR}_*, \overline{\HMR}_*.
$$
For convenience we denote them as $\HMR^{\circ}$, for $\circ \in \{\vee, \wedge, -\}$.
%that fit into the long exact sequence
%\[\begin{tikzcd}
%	{} & {\widecheck{\HMR}(K)} & {\widehat{\HMR}(K)} & {\overline{\HMR}(K)} & {\widehat{\HMR}(K)} & {}
%	\arrow["{j_*}", from=1-2, to=1-3]
%	\arrow["{p_*}", from=1-3, to=1-4]
%	\arrow["{i_*}", from=1-1, to=1-2]
%	\arrow["{i_*}", from=1-4, to=1-5]
%	\arrow["{j_*}", from=1-5, to=1-6]
%\end{tikzcd}\]
The main result is the following.
\begin{thm}
\label{thm:main1}
Let $(K_0, K_1, K_2)$ be an unoriented skein triple of links.
Then there is an exact triangle 
\[\begin{tikzcd}
	{\HMR^{\circ}_*(K_2)} && {\HMR^{\circ}_*(K_1)} \\
	& {\HMR^{\circ}_*(K_0)}
	\arrow[from=1-1, to=1-3]
	\arrow[from=1-3, to=2-2]
	\arrow[from=2-2, to=1-1]
\end{tikzcd}\]
for each flavour $\circ \in \{\vee,\wedge,-\}$ of real monopole Floer homologies.
\end{thm}
Unoriented skein triangles were constructed in e.g. \cite{Manolescu2007triangle,KMunknot2011} for various versions of Floer homologies of knots and links.
They are closely related to surgery exact triangles e.g. \cite{Floer1995triangle,OzSz2004,KMOS2007} for Floer homologies of 3-manifolds via the branched double cover.
Floer homologies of branched double covers as link invariants were explored in e.g.  \cite{OzSz2005HFdc,Bloom2011,Scaduto2015}.
Our main theorem is a refinement of this idea.
The key observation behind the proof of Theorem~\ref{thm:main1} is that the reducible solutions involved in the proof in \cite{KMOS2007} of the surgery exact triangle in the monopole Floer homology can be made compatible with real structures.

\begin{rem}
    In this paper, we specialize to classical links in $S^3$ to keep notations minimal.
    However, we see no difficulty in extending the proof of Theorem~\ref{thm:main1} (which is local in nature) to more general links in 3-manifolds.
    Since real monoopole Floer homologies were defined for manifolds with involutions \cite{ljk2022}, $\HMR^{\circ}$ can be defined for a link $K$ in a 3-manifold $Y$ such that
    \[
        [K] = 0 \in H_1(Y,\mathbb Z/2).
    \]
    This assumption guarantees the existence of double branched covers along $K$, and Theorem~\ref{thm:main1} is expected to hold for this class of links in 3-manifolds.
    See Remark~\ref{rem:generalK} for more details.
\end{rem}

Furthermore, we define an absolute grading for the real monopole Floer homology of links and introduce a real version of Fr\o yshov invariant \cite{FroyshovMonopoleQH3}.
We give a basic version of Fr\o yshov's inequality, and compute the invariant in some simple examples.
Knot invariants from Fr\o yshov-type invariants on branched covers were previously studied in e.g. \cite{OzSz2005HFdc,ManolescuOwens2007}.
Analogues of our \emph{real} Fr\o yshov invariant were introduced in \cite{KMT2023} using Seiberg-Witten-Floer spectra.
It is natural to expect an isomorphism between the real monopole Floer homologies and the real Seiberg-Witten-Floer homologies of \cite{KMT2023}.

%\begin{thm}
%	There is an isomorphism of Floer homology groups
%	\begin{equation*}
%	\HMR^{\circ}_{\bullet}(K_2) \cong H_{\bullet}(\text{Cone}(\HMR^{\circ}_{\bullet}(K_1) \to \HMR^{\circ}_{\bullet}(K_0))),
%	\end{equation*}
%	where $\text{Cone}$ denotes the mapping cone complex, and the map is $\HMR^{\circ}_{\bullet}(S_{10}): \HMR^{\circ}_{\bullet}(K_1) \to \HMR^{\circ}_{\bullet}(K_0)$ induced by the skein cobordism $S_{10}$.
%\end{thm}

\subsection{Organization of sections}
\hfill \break
We review the basics of real monopole Floer homology in Section~\ref{sec:HMRforLinks}.
We define the absolute $\mathbb Q$-grading and Fr\o yshov invariant in Section~\ref{sec:Froyshov}.
We discuss the topological aspects of the skein triangle in Section~\ref{sec:top_skein} and give the proof of the main theorem in Section~\ref{sec:prof_skein}.
Some examples are discussed in Section~\ref{sec:exmp}.

\section{Real Monopole Floer Homologies for Links}
\label{sec:HMRforLinks}
We begin with the general setup of real monopole Floer homologies $\HMR^{\circ}(Y,\upiota)$ as invariants of 3-manifolds with involutions $(Y,\upiota)$.
\subsection{Formal properties}
\hfill \break
Let $Y$ be a closed oriented 3-manifold and $\upiota:Y \to Y$ be a smooth, orientation-preserving involution.
The real monopole Floer homologies are $\mathbb F_2$-vector spaces
\begin{equation*}
	\widecheck{\HMR}_*(Y, \upiota), \quad
	\widehat{\HMR}_*(Y, \upiota), \quad \overline{\HMR}_*(Y, \upiota).
\end{equation*}
Each of the group decomposes into a direct sum over real \emph{spin\textsuperscript{c} structures} $(\mathfrak s, \uptau)$, which play the r\^{o}le of  spin\textsuperscript{c} structures in $\HM$:
\begin{equation*}
	\HMR^{\circ}_{*}(Y,\upiota) = \bigoplus_{\mathfrak s} \HMR^{\circ}_{*}(Y,\upiota;\mathfrak s,\uptau).
\end{equation*}
\begin{defn}
	A \emph{real} spin\textsuperscript{c} structure is a pair $(\mathfrak s, \uptau)$, such that $\mathfrak s = (S,\rho)$ is a spin\textsuperscript{c} structure on $Y$, and $\uptau:S \to S$ is an anti-linear involutive lift of $\upiota$, which is \emph{compatible with $\mathfrak s$} in the sense that $\langle \uptau(s_1), \uptau(s_2)\rangle_{\upiota(y)}=\overline{\langle s_1, s_2\rangle_y}$ and 
	 \begin{equation*}
		\rho(\upiota_* \xi)\uptau(\Phi_{\upiota(y)}) = \uptau(\rho(\xi)\Phi_y),
	\end{equation*}
	for any $y \in Y$, any vector field $\xi$ on $Y$, and any spinor $\Phi \in \Gamma(S)$.
 An anti-linear involution $\uptau:S \to S$ covering $\upiota: Y \to Y$ is a \emph{real structure} of $S$.
\end{defn}
Let $K \subset S^3$ be a link and $\mathfrak s$ be a spin\textsuperscript{c} structure on its branched double cover  $\Sigma_2(S^3,K)$, equipped with the covering involution $\upiota_{deck}$.
While compatible real structures do not always exist on a general 3-manifold with involution, any spin\textsuperscript{c} structure on $\Sigma_2(S^3, K)$ supports a unique compatible real structure, up to equivalence.
Applying $\HMR^{\circ}$ to $(\Sigma_2(S^3, K),\upiota_{deck})$ yields
\begin{equation*}
	\widecheck{\HMR}_*(K;\mathfrak s), \quad
	\widehat{\HMR}_*(K;\mathfrak s), \quad 
    \overline{\HMR}_*(K;\mathfrak s),
\end{equation*}
which are isotopy invariants of $K$.

Each group admits a grading ``$*$'' by set $\mathbb J(\mathfrak s)$ with a $\mathbb Z$-action.
We use ``$\bullet$'' to denote the completion (to be defined in Section~\ref{subsec:gauge_constr}) of the grading ``$*$'' .
Then $\HMR^{\circ}_{\bullet}(K;\mathfrak s)$ is a $\mathcal R_n$-module, where $n$ is the number of components of $K$ and $\mathcal R_n$ is the ring
\begin{equation*}
	\mathcal R_n = \frac{\mathbb F_2[[\upsilon_1, \dots, \upsilon_n]]}{\upsilon_i^2 = \upsilon_j^2},
\end{equation*}
such that each $\upsilon_i$ has degree $(-1)$.
Denote the squares $\upsilon_i^2$ by $U$.
The ring $\mathcal R_n$ is related to the cohomology of the Picard torus.

Let $K_-, K_+$ be two links.
A \emph{cobordism} $\Sigma: K_- \to K_+$ is a properly embedded, possibly non-orientable surface in $[0,1] \times S^3$, whose boundary consists of $\{0\} \times K_-$ and $\{1\} \times K_+$.
The functoriality of $\HMR^{\circ}$ states that there exists an $\mathbb F_2[[U]]$-module map
\begin{equation*}
	\HMR^{\circ}_{\bullet}(\Sigma):\HMR^{\circ}_{\bullet}(K_-) \to \HMR^{\circ}_{\bullet}(K_+)
\end{equation*}
satisfying the composition law.
This is a special case of the functoriality of $
\HMR^{\circ}$ for 3-manifolds with involutions.
\begin{rem}
    We only work with double branched covers so a ``branched cover'' will always mean a double branched cover.
\end{rem}

\subsection{The gauge-theoretic construction}
\label{subsec:gauge_constr}
\hfill \break
Let $(Y,\upiota)$ be a 3-manifold with involution and $(\mathfrak s, \uptau)$ be a real spin\textsuperscript{c} structure.
Let  $g$ be an $\upiota$-invariant Riemannian metric.
The \emph{real} Seiberg-Witten configuration space is the space of pairs $(A,\Phi)$ such that
\begin{itemize}[leftmargin=*]
\item $A$ is a $\uptau$-invariant spin\textsuperscript{c} connection for $\mathfrak s$, in the sense that
\begin{equation*}
	\nabla_A = \uptau \circ \nabla_A \circ \uptau.
\end{equation*}
The space $\mathcal A(Y,\mathfrak s,\uptau)$  of $\uptau$-invariant spin\textsuperscript{c}-connections is affine over $1_S \otimes \Omega^1(Y;i\reals)^{-\upiota^*}$, consisting of the $\upiota^*$-anti-invariant imaginary-valued 1-forms.

\item $\Phi$ is $\uptau$-invariant, in the sense that
\begin{equation*}
		\Phi_y =\uptau(\Phi_{\upiota(y)}).
	\end{equation*}
Denote by $\Gamma(S)^{\uptau}$ the real subspace of $\uptau$-invariant spinors.
\end{itemize}
The \emph{real Seiberg-Witten configuration space} is the product  space
\begin{equation*}
	\mathcal C(Y, \mathfrak s,\uptau)=
	\mathcal A(Y,\mathfrak s,\uptau) \times
	\Gamma(S)^{\uptau}.
\end{equation*}
The gauge group $\mathcal G(Y,\upiota)$ is the subgroup of $\upiota$-invariant automorphisms
\begin{equation*}
	\mathcal G(Y,\upiota) = \{\bar u(\upiota(y)) = u(y)\},
\end{equation*}
acting on $\mathcal C$ by
\[u \cdot (A,\Phi) = (A - u^{-1}du, u\Phi).\]
The subgroup of constant gauge transformation is $\{\pm 1\}$.
A configuration $(A,
\Phi)$ is \emph{reducible} if $\Phi = 0$, and the space of reducible configurations is precisely the fixed-point set of the $\{\pm1\}$-action.
The space of equivalence classes of configurations is defined to be the quotient
\begin{equation*}
	\mathcal B(Y, \mathfrak s,\uptau) =
	\mathcal C(Y, \mathfrak s,\uptau)/\mathcal G(Y,\upiota).
\end{equation*}

The Chern-Simons-Dirac (CSD) functional is a $\uptau$-invariant function $\mathcal L: \mathcal C(\mathfrak s) \to \reals$, whose formal $L^2$-gradient vector field is given by
\begin{equation*}
	\left(\frac{1}{2}*F_{B^t} + \rho^{-1}(\Psi\Psi^*)_0,
	D_B\Psi
	\right) \in \Omega^1(Y;i\mathbb{R}) \oplus \Gamma(S),
\end{equation*}
where $(\Psi\Psi^*)_0 = \Psi\Psi^* - \frac12|\Psi|^2$ is a traceless,  self-adjoint endomorphism of the spinor bundle, and hence lies in the image of $\rho(\Omega^1(Y;i\reals))$.
The gradient flow equations are exactly the Seiberg-Witten equations over the cylinder $\reals \times Y$.

The idea of Floer homology is to study the Morse theory of the vector field $\grad \mathcal L$.
In our setup, the Morse theory happens only on the invariant (\emph{real}) subspace of the ordinary Seiberg-Witten configuration space.
To deal with reducibles, we follow Kronheimer-Mrowka's approach by blowing up the configuation space along the reducibles: set
\begin{equation*}
	\mathcal C^{\sigma}(Y,\mathfrak s, \uptau) = \{(B, r, \psi): r \in \reals \text{, } (B,\psi) \in \mathcal C(Y,\mathfrak s, \uptau)  \text{, and } \|\psi\|_{L^2(Y)}=1 \}.
\end{equation*}

The $\{\pm 1\}$ gauge transformations act freely on $	\mathcal C^{\sigma}(Y,\mathfrak s, \uptau)$, and topologically the blowup process replaces the locus $\mathcal A(Y,\mathfrak s,\uptau) \times \{0\}$ by the unit sphere of the $L^2$-spinors.
The space of gauge-equivalence classes
\[
\mathcal B^{\sigma}(Y,\mathfrak s, \uptau)= \mathcal C^{\sigma}(Y,\mathfrak s, \uptau)/\mathcal G(Y,\upiota)
\] 
has the structure of a manifold with boundary $\del \mathcal B^{\sigma}(Y,\mathfrak s, \uptau)$.
The boundary is precisely the reducible blown-up configuration space, which is homotopy-equivalent to $\mathbb{RP}^{\infty}$.

We study the blown-up version $(\grad \mathcal L)^{\sigma}$ of the formal CSD gradient
that agrees with the pullback of $\grad \mathcal L$ under the blow-down map $\pi: (B,r,\psi) \mapsto (B,r\psi)$ over the irreducible.
While this is not the gradient with respect to any natural metric, $(\grad \mathcal L)^{\sigma}$ retains many important properties of a formal gradient.
The critical points of $(\grad \mathcal L )^{\sigma}$ come in three sorts:\begin{itemize}[leftmargin=*]
	\item the \emph{interior} critical points, i.e. the irreducibles,
	\item the \emph{boundary-stable} reducibles, at which the Hessian is positive along the normal direction to the boundary, and
	\item the \emph{boundary-unstable} reducibles, where the Hessian is negative  along the normal direction.
\end{itemize}
In fact, we perturb  $(\grad \mathcal L)^{\sigma}$ to achieve transversality of moduli spaces of trajectories.
The usual transversality fails always when the trajectories go from a boundary-stable critical point to a boundary-unstable one (in this case, we refer to the moduli space as \emph{boundary obstructed}).
This phenomenon is present already in finite dimensional Morse theory of manifolds with boundaries.

Denote by $C^o, C^s, C^u$ the $\mathbb F_2$-vector spaces generated by interior, boundary-stable, and boundary-unstable critical points, respectively.
By counting zero-dimensional spaces of trajectories, we define linear maps
\begin{equation*}
	\del^o_o: C^o \to C^o, \quad 
	\del^o_s: C^o \to C^s \quad 
	\del^u_o: C^u \to C^o \quad 
	\del^u_s: C^o \to C^o.
\end{equation*}
Moreover, counting zero-dimensional space of reducible trajectories, we define maps
\begin{equation*}
	\bar \del^s_s: C^s \to C^s, \quad 
	\bar \del^u_s: C^u \to C^s \quad 
	\bar \del^s_u: C^s \to C^u \quad 
	\bar \del^u_u: C^u \to C^u.
\end{equation*}
We define three chain complexes
\begin{equation*}
	\bar C_k = C^s_k \oplus C^u_{k+1},\quad
	\check C_k = C^o_k \oplus C^s_k \quad
	\hat C_k = C^o_k \oplus C^u_k,
\end{equation*}
given by differentials
\begin{equation*}
	\bar\del = \begin{pmatrix}
		\bar\del^s_s & \bar\del^u_s\\
		\bar\del^s_u & \bar\del^u_u
	\end{pmatrix}, \quad 
	\check\del = \begin{pmatrix}
		\del^o_o & -\del^u_o\bar \del^s_u \\
		\del^o_s & \bar\del^s_s - \del^u_s\bar\del^s_u
	\end{pmatrix},
	\quad 
	\hat\del = \begin{pmatrix}
		\del^o_o & \del^u_o\\
		-\bar\del^s_u\bar\del^o_s & -\bar\del^u_u -\bar\del^s_u\del^u_s
	\end{pmatrix}
\end{equation*}
respectively.
We define
\begin{equation*}
	\widecheck{\HMR}_*(Y,\upiota;\mathfrak s, g,\uptau) = H_*(\check C,\check \del), \quad
	\widehat{\HMR}_*(Y,\upiota;\mathfrak s, g,\uptau) = H_*(\hat C,\hat \del), \quad
	\overline{\HMR}_*(Y,\upiota;\mathfrak s, g,\uptau) = H_*(\overline C,\bar \del).
\end{equation*}
The long exact sequence for the three Floer homology groups
\begin{equation*}
	\begin{tikzcd}
		{...} \ar[r,"i_*"] 
		& \widecheck{\HMR}_*(Y,\upiota;\mathfrak s,\uptau)	 \ar[r,"j_*"] 
		& \widehat{\HMR}_*(Y,\upiota;\mathfrak s, \uptau) \ar[r,"p_*"]
		& \overline{\HMR}_*(Y,\upiota;\mathfrak s, \uptau) \ar[r,"i_*"]
		& \widehat{\HMR}_*(Y, \upiota;\mathfrak s, \uptau) \ar[r,"j_*"]
		&{...}
	\end{tikzcd}
\end{equation*}
is defined by the (anti-)chain maps
\begin{equation*}
	i = \begin{pmatrix}
		0 & \del^u_o \\
		1 & -\del^u_s
	\end{pmatrix} :\bar C \to \check C,
	\quad 
	j = \begin{pmatrix}
		1 & 0 \\
		0 & -\delbar^u_s
	\end{pmatrix} :\check  C \to \hat  C, 
	\quad 
	p = \begin{pmatrix}
		\del^o_s & \del^u_s \\
		0 & 1
	\end{pmatrix} :\check  C \to \bar  C.
\end{equation*}

Let $\mathbb J(Y,\mathfrak s, \uptau)$ be the grading set. There is a $\mathbb Z$-action, denoted as
\begin{equation*}
	j \mapsto j+ n
\end{equation*}
for $j \in J(Y,\mathfrak s, \uptau)$.
We write $\HMR^{\circ}_*(Y,\upiota;\mathfrak s, g,\uptau)$ for the $J(Y,\mathfrak s, \uptau)$-graded real monopole Floer homology groups.
%A cobordism map $\HMR^{\circ}(W,\upiota_W)$ may have infinitely many entries in positive degrees, and the image of a single element in $\HMR^{\circ}(Y_-,\upiota_-)$ may not be contained in any finitely generated subgroup.
To define cobordism maps, we complete the Floer homology group with respect to the following filtration
\begin{equation*}
	\HMR^{\circ}_{\bullet}(Y,\upiota)[n] = \bigoplus_{\alpha} \bigoplus_{m \ge n}\HMR^{\circ}_{j_{\alpha}-m}(Y,\upiota),
\end{equation*}
where $\{j_{\alpha}: \alpha \in A\}$ is a choice of representatives of free $\mathbb Z$-orbits.

\subsection{Cobordism maps}
\hfill \break
Let $(X,\upiota)$ be a 4-manifold with boundary, equipped with an involution $\upiota$ that preserves the boundary components.
\begin{defn}
	A \emph{real} spin\textsuperscript{c} structure is a pair $(\mathfrak s, \uptau)$, where $\mathfrak s = (S^{\pm},\rho)$ is a spin\textsuperscript{c} structure on $X$, and $\uptau:S^{\pm} \to S^{\pm}$ is an anti-linear involution that covers $\upiota$ over the base, and \emph{compatible with $\mathfrak s$} in the sense that $\langle \uptau(s_1), \uptau(s_2)\rangle_{\upiota(x)}=\overline{\langle s_1, s_2\rangle_x}$ and 
	 \begin{equation*}
		\rho(\upiota_* \xi)\uptau(\Phi_{\upiota(x)}) = \uptau(\rho(\xi)\Phi_x),
	\end{equation*}
	for any $x \in X$, any vector field $\xi$ on $X$, and any spinor $\Phi \in \Gamma(S^+)$.
\end{defn}
Let $X^*$ be the manifold with cylindrical ends obtained by adjoining cylindrical ends $[0,\infty) \times \del X$ to $X$, where the involution $\upiota_X:X^* \to X^*$ acts trivially on the $[0, \infty)$ factor.
We perturb the Seiberg-Witten equations on a compact part on the cylinder.
The moduli space of the perturbed Seiberg-Witten solutions, asymptotic to a vector of critical points $[\boldsymbol{\mathfrak b}]$ on the cylindrical points 
\begin{equation*}
	M(X^*,\mathfrak s_X, \uptau_X;[\mathfrak b]) 
\end{equation*}
is regular, and admits compactifications by broken trajectories: 
\begin{equation*}
	M^+(X^*,\mathfrak s_X, \uptau_X;[\mathfrak b]) \text{ and }
	\bar{M}(X^*,\mathfrak s_X, \uptau_X;[\mathfrak b]),
\end{equation*}
where former compactification is finer than the latter (see \cite[Section~12]{ljk2022}).

We also consider parameterized moduli spaces.
Given a finite-dimensional manifold $P$, we  define
\begin{equation*}
	M(X^*,\mathfrak s_X, \uptau_X, [\mathfrak b])_P
	= 
	\bigcup_{p \in P} \ \{p\}
 \times M(X^*,\mathfrak s_X, \uptau_X, [\mathfrak b])_p.
\end{equation*}
Let $z$ be a connected component of $	M(X^*;\mathfrak s_X, \uptau_X, [\mathfrak b])_P$.
We denote the expected dimension of the moduli spaces as
\[
\gr_z (X^*,\mathfrak s_X, \uptau_X;[\mathfrak b])_P
\]
As before, we compactify the moduli spaces fibrewise to obtain
\begin{equation*}
	M^+(X^*,\mathfrak s_X, \uptau_X;[\mathfrak b])_P \text{ and }
	\bar{M}(X^*,\mathfrak s_X, \uptau_X;[\mathfrak b])_P.
\end{equation*}

Next, we assume $(W,\upiota_W):(Y_-,\upiota_-) \to (Y_+, \upiota_+)$ is a cobordism between 3-manifolds with involutions.
Fix a connected component $z$ of the blown-up configuration space over $W$, we denote the (coarse) compactified moduli spaces of perturbed Seiberg-Witten equations over cobordism as
\[
	\bar M_z([\mathfrak a],W^*,\upiota_W,[\mathfrak b])\text{ and }\bar M_z^{\text{red}}([\mathfrak a],W^*,\upiota_W,[\mathfrak b])
\]
where $[\mathfrak a]$ is a critical point on the incoming end $Y_-$ and $[\mathfrak b]$ is a critical point on the outgoing end $Y_+$.

The cobordism map involves the following matrix entries
\begin{equation*}
	m^o_o, m^o_s, m^u_o, m^u_s, \bar m^s_s, \bar m^u_u, \bar m^s_u, \bar m^u_s,
\end{equation*}
where the overlined versions count reducible solutions ``$M^{\red}$''. 
For example, $m^u_s: C^u_{\bullet}(Y_-,\upiota_-) \to C^s_{\bullet}(Y_+, \upiota_+)$ is given by counts of zero-dimensional moduli spaces:
\begin{equation*}
	\sum_{[\mathfrak b] \in \mathfrak C^s} \sum_{z} \# M_z([\mathfrak a], W^*, \upiota_W, [\mathfrak b]) \cdot [\mathfrak b].
\end{equation*}
The three cobordism maps
\begin{align*}
		\overline{\HMR}(W, \upiota_W):
		\overline{\HMR}_{\bullet}(Y_-, \upiota_-) 
		&\to \overline{\HMR}_{\bullet}(Y_+, \upiota_+),\\
		\widecheck{\HMR}(W, \upiota_W):
		\widecheck{\HMR}_{\bullet}(Y_-, \upiota_-) 
		&\to \widecheck{\HMR}_{\bullet}(Y_+, \upiota_+),\\
		\widehat{\HMR}(W, \upiota_W):
		\widehat{\HMR}_{\bullet}(Y_-, \upiota_-) 
		&\to \widehat{\HMR}_{\bullet}(Y_+, \upiota_+),
\end{align*}
are defined at chain level by the matrices
\begin{equation*}
	\bar m = \begin{pmatrix}
		\bar m^s_s && \bar m^u_s\\
		\bar m^s_u && \bar m^u_u
	\end{pmatrix}, \quad 
	\check m = \begin{pmatrix}
		m^o_o && -m^u_o \bar\del^s_u(Y_-) - \del^u_o(Y_+)\bar m^s_u\\
		m^o_s && \bar m^s_s - m^u_s\bar\del^s_u(Y_-,) -\del^u_s(Y_+)\bar m^s_u
	\end{pmatrix},
\end{equation*}
and
\begin{equation*}
	\hat m = \begin{pmatrix}
		m^o_o && m^u_o\\
		\bar m^s_u \del^o_s(Y_-)\sigma  - \bar\del^s_u(Y_+)m^o_s 
		&& 
			\bar m^u_s \sigma + \bar m^s_u \del^u_s (Y_-) \sigma - \bar\del^s_u(Y_+)m^u_s
	\end{pmatrix}.
\end{equation*}
We put brackets $(Y_{\pm})$ to distinguish the differentials on the two 3-manifolds (with involutions).
The approach in \cite{ljk2022}, following \cite{KMbook2007}, simultaneously evaluates a cohomology class of the configuration space $\mathcal B^{\sigma}(W^*,\upiota_W)$. In the case of a cylinder with a cohomology class, we obtain module structures of $\HMR^{\circ}$.

%Finally, we make a remark about our notations in the context of $\HMR$ for classical links. 
For branched covers of $[0,1] \times S^3$ and $B^4$, there is no ambiguity in the choices of the compatible real structures (see \cite[Section~3.1]{ljk2022}).
We will therefore write, for instance,
\begin{equation*}
	M_z([\mathfrak a], \Sigma, [\mathfrak b]), \quad
	M_z(\Sigma, [\mathfrak a])
\end{equation*}
for Seiberg-Witten moduli spaces over the branched covers along $\Sigma$ when $\Sigma$ is a subset of $[0,1] \times S^3$ or $B^4$.
More generally, $\Sigma$ can be a properly embedded surface in multiply-punctured $S^4$.
\begin{rem}
\label{rem:generalK}
The functoriality for $\HMR^{\circ}$ for 2-divisible links in 3-manifolds requires additional data.
To define $\HMR^{\circ}$, both the choice of real lifts and the choice of double branched covers are not necessarily unique.
The real monopole Floer homology group for a link $K$ in $Y$ is defined as the following direct sum
\[
    \HMR^{\circ}(K) = \bigoplus_{\mathbf Y, \mathfrak s, \uptau} \HMR^{\circ}(\mathbf Y, \mathfrak s, \uptau),
\]
where $\mathbf Y$ is a double branched cover of $Y$ along $K$, where $\mathfrak s$ is a spin\textsuperscript{c} structure, and $\uptau$ is a compatible real structure.
It follows that the cobordism maps must be equipped with the data of a branched cover, a spin\textsuperscript{c} structure, a real lift satisfying the natural compatibility conditions with the boundary data.
\end{rem}

\subsection{Examples of $\HMR^{\circ}(K)$}
\hfill \break
For the later sections on exact triangles we recall some elementary examples in \cite{ljk2022}, which all happen to admit (invariant) positive scalar curvature metrics.
This is not the case in general, and we refer the readers to \cite[Section~12]{ljk2022} for discussions on torus knots and Montesinos knots.
\begin{exmp}
	\label{exmp:U1}
	Let $\text{U}_1$ be the unknot. 
    A concrete model of the branched cover is the unit sphere $S^3 \subset \C^2$, where the covering involution is given by the conjugation action
	\[
	\upiota(z_1,z_2) = (\bar z_1, \bar z_2)
	\]
	over $\C^2$.
	The induced metric is $\upiota$-invariant and has positive scalar curvature.
	Let $\mathfrak s = (S,\rho)$ be the unique spin\textsuperscript{c} structure on $S^3$ and $\uptau: S \to S$ be a compatible real structure.
	By positive scalar curvature, there is no irreduible critical point and only one unique critical point $[B,0]$ of the unperturbed Chern-Simons-Dirac functional $\mathcal L$, where $B$ is a spin\textsuperscript{c} connection.
We perturb the the Dirac operator $D_B$
	to obtain an operator $D_{\mathfrak q, B}$ whose spectrum is simple and contains no zero.
 
	Label the eigenvalues $\{\lambda_i:i \in \mathbb Z\}$ of $D_{\mathfrak q, B}$ in increasing order, so that $\lambda_0$ is the smallest positive eigenvalue.
	Let $\{[\mathfrak a_i]\}$ be the corresponding reducible critical points of $(\grad \pertL)^{\sigma}$.
	There exists no irreducible trajectories on cylinders, and the moduli spaces
	\[
		M_{z}([\mathfrak a_i], [\mathfrak a_j])
	\]
	can be computed explicitly (see \cite[Proposition~12.2]{ljk2022} and \cite[Proposition~14.6.1]{KMbook2007}).
	In particular, one can show that all differentials vanish.
	The real monopole Floer homologies are isomorphic to their corresponding chain complexes, and $\upsilon$ acts by
	\[
		\upsilon [\mathfrak a_i] = [\mathfrak a_{i-1}].
	\]
	It follows that
	\begin{align*}
		\overline{\HMR}_*(\text{U}_1;\mathfrak s) 
		&\cong \mathbb F_2[\upsilon^{-1},\upsilon],\\
		\widehat{\HMR}_*(\text{U}_1;\mathfrak s) 
		&\cong \mathbb F_2[\upsilon],\\
		\widecheck{\HMR}_*(\text{U}_1;\mathfrak s) 
		&\cong \mathbb F_2[\upsilon^{-1},\upsilon]/\mathbb F_2[\upsilon],
	\end{align*}
\end{exmp}
\begin{exmp}
\label{exmp:U2}
	Let $\text{U}_2$ be the $2$-component unlink. 
	The double branched cover is $S^1 \times S^2$ and the covering involution $\upiota$ is given by $\upiota_1 \times \upiota_2$ where $\upiota_1: S^1 \to S^1$ reflects the circle fixing two points, and $\upiota_2:S^2 \to S^2$ is an orientation-reversing involution that swaps two hemispheres and fixes a great circle.
	The Riemmannian metric can be chosen $\upiota$-invariantly and have positive scalar curvature.
	Consider the unique torsion real spin\textsuperscript{c} structure $\mathfrak s_0$.
	Then there exists no irreducible critical points to the CSD functional and remains so under small perturbations.
	We perturb the CSD functional using a Morse function on the invariant flat connections over $S^1 \times S^2$ which is homeomorphic to a circle.
	Assume the Morse function has two critical points $\alpha^1$ and $\alpha^0$ of index $1$ and $0$, respectively. 
	The perturbed CSD critical points consist of two towers:
	\begin{equation*}
		\{e^1_i\}, \{e^0_i\},
	\end{equation*}
	where $i \in \mathbb Z$ and $\ind(e^1_i) = \ind(e^0_i) + 1$, lying above $[\alpha^1,0]$ and $[\alpha^0,0]$ in the non-blown-up configuration space.
	The \emph{real} indices of elements on the same tower satisfy $\ind(e^{\mu}_{i+1}) = \ind(e^{\mu}_i) + 1$, instead of a difference of two in $\HM$.
	Denote by $\mathfrak a^{\mu}_i$ the element  corresponding to $e^{\mu}_i$ in  homology.
	One can show that the differentials in $\HMR^{\circ}$ are all zero, and $\upsilon \mathfrak a^{\mu}_i = a^{\mu}_{i-1}$. We conclude that
	\begin{align*}
		\overline{\HMR}_*(U_2;\mathfrak s_0,\uptau_0) 
		&\cong \mathbb F_2[\upsilon^{-1},\upsilon]  \oplus \mathbb F_2[\upsilon^{-1},\upsilon]\langle +1\rangle \\
		\widehat{\HMR}_*(U_2;\mathfrak s_0,\uptau_0) 
		&\cong \mathbb F_2[\upsilon] \oplus \mathbb F_2[\upsilon]\langle +1\rangle ,\\
		\widecheck{\HMR}_*(U_2;\mathfrak s_0,\uptau_0) 
		&\cong \mathbb F_2[\upsilon^{-1},\upsilon]/\mathbb F_2[\upsilon] \oplus
		\left(\mathbb F_2[\upsilon^{-1},\upsilon]/\mathbb F_2[\upsilon]\right)\langle +1\rangle.
\end{align*}
We use the angled bracket to denote shifting of grading.
See \cite[Corollary~14.4]{ljk2022} for more details.
\end{exmp}

\begin{exmp}
\label{exmp:rational}
Let $K(p,q)$ be the $(p,q)$-rational knot, for $p > q$ coprime and $p > 2$.
The branched cover is $L(p,q)$, and can be thought of as the quotient space of the unit sphere $S^3 \subset \C^2$ under 
\begin{equation*}
	(z_1,z_2) \mapsto (e^{2\pi i/p}z_1,e^{2\pi iq/p}z_2).
\end{equation*}
The covering involution is induced by the conjugation action $(z_1,z_2) \mapsto (\bar z_1,\bar z_2)$ upstairs.
There is a unique spin structure and $p$ spin\textsuperscript{c} structures.
The real Floer homology groups are isomorphic to that of the unknot:
\begin{align*}
		\overline{\HMR}_*(K(p,q);\mathfrak s) 
		&\cong \mathbb F_2[\upsilon^{-1},\upsilon],\\
		\widehat{\HMR}_*(K(p,q);\mathfrak s) 
		&\cong \mathbb F_2[\upsilon],\\
		\widecheck{\HMR}_*(K(p,q);\mathfrak s) 
		&\cong \mathbb F_2[\upsilon^{-1},\upsilon]/\mathbb F_2[\upsilon].
\end{align*}
In particular, the Floer homology groups are completely determined by their real Fr\o yshov invariant, to be defined in Section~\ref{sec:Froyshov}.
This example was discussed in \cite[Section~14.2]{ljk2022}. 
\end{exmp}
\section{Absolute $\mathbb Q$-grading and the Fr\o yshov invariant}
\label{sec:Froyshov}
Let $K_-$ and $K_+$ be two links in $S^3$ and $Y_{\pm}$ be the corresponding branched covers.
Let $\Sigma: K_- \to K_+$ be a link cobordism, and $W$ be the branched cover of $[0,1] \times S^3$ along $\Sigma$.
Let $\iota(\Sigma)$ be the number such that $-\iota(\Sigma)$ is the index of the linearized abelian ASD operator over weighted Sobolev spaces:
\begin{equation*}
	d^* \oplus d^+:e^{\delta w} L^2_1(W^*; \Lambda^1)^{-\upiota^*}
	\to e^{\delta w} L^2_1(W^*; \Lambda^0 \oplus \Lambda^+)^{-\upiota^*},
\end{equation*}
where $\delta$ is a small positive weight, and $w$ is a function that restricts to the cylindrical end as the coordinate $t$.

We express $\iota(\Sigma)$ in terms of familiar topological quantities.
For $\ell \in \{0,1, +\}$, the Sobolev space of $\ell$-forms admits a decomposition by $(\pm 1)$-eigenspaces of the operator $\upiota^*$:
\begin{equation*}
	L^2_k(W^*;\Lambda^{\ell}) 
	=
	L^2_k(W^*;\Lambda^{\ell})^{\upiota^*} \oplus 
	L^2_k(W^*;\Lambda^{\ell})^{-\upiota^*}.
\end{equation*}
The $\upiota_W^*$-invariant spaces are naturally isomorphic to Sobolev spaces over the quotient $\reals \times S^3$:
\begin{equation*}
	L^2_k(W^*;\Lambda^{\ell})^{\upiota_W^*} \cong 
	L^2_k(W^*/\upiota_W;\Lambda^{\ell}) 
	\cong L^2_k(\reals \times S^3;\Lambda^{\ell}).
\end{equation*}
The ASD operator on the trivial bundle with a small positive weight over $\reals \times S^3$ is invertible.
Thus
\begin{equation*}
	\iota(\Sigma) 
	= \frac{\chi(W) + \sigma(W) + b_1(Y_+) - b_1(Y_-)}{2}.
\end{equation*}
%\\
%&= b^+(W) - b^1(W) + b^0(Y_+) + b^1(Y_+).
%\end{align*}
%The two formulae can be seen to be equivalent by the cohomology exact sequence of pairs $(W,\del W)$.
In particular, the number $\iota(\Sigma)$ agrees with $\iota(W)$ defined in \cite[Section~24.1]{KMOS2007}.
The following quantities of $W$ can be computed from $\Sigma:K_- \to K_+$ (see e.g. \cite{KauffmanTaylor1976})
\begin{align*}
	\sigma(W) &= \sigma(K_+) - \sigma(K_-) - \frac{\Sigma \cdot \Sigma}{2},
	\\
	\chi(W) &= b_1(\Sigma) - b_0(\Sigma).
\end{align*}
The self-intersection number $\Sigma \cdot \Sigma$ (or $\Sigma^2$) is well-defined even for non-orientable surfaces.
If $\Sigma$ is connected, then (cf. \cite{LeeWeintraub1995})
\begin{align*}
	b_1(W) &= 0,\\
	b^+(W) &= \frac{1}{2}\left(b_1(\Sigma)  + \sigma(K_+) - \sigma(K_-)-\frac{\Sigma \cdot \Sigma }{2}\right).
\end{align*}

Suppose $\Sigma$ is more generally a properly embedded surface in 
\[S^4 - \{\Ball_i: 1 \le i \le n+m\},\]  an $(n+m)$-times punctured sphere with $n$ incoming ends and $m$ outgoing ends. 
Let $W$ be the branched cover along $\Sigma$,
and let $-\iota(\Sigma)$ be the index over the $(-\upiota^*)$-invariant ASD-complex.
Then
\[
	-\iota(W) = \ind(d^* \oplus d^+) = \ind^{\upiota^*}(d^* \oplus d^+) + \ind^{-\upiota^*}(d^* \oplus d^+) 
	= (n-1) - \iota(\Sigma),
\]
as over the punctured sphere we have $\ind(d^* \oplus d^+) = n-1$.

\begin{defn}
	Let $\mathfrak s$ be a torsion spin\textsuperscript{c} structure on $Y = \Sigma_2(S^3, K)$, and let $[\mathfrak a]$ be a critical point on $Y$.
	Let $\Sigma: \text{U}_1 \to K$ be a link cobordism and $W: S^3 \to Y$ be the corresponding branched cover of $[0,1] \times S^3$.
	Let $z$ be a $W$-path (in the sense of \cite[Definition~13.6]{ljk2022}) from $[\mathfrak a_0]$ to $[\mathfrak a]$.
	We define a rational number $\gr^{\mathbb Q}([\mathfrak a])$ by the formula
	\begin{equation*}
		\gr^{\mathbb Q}([\mathfrak a])=
		-\gr_z([\mathfrak a_0], \Sigma, [\mathfrak a])+ \frac{1}{8}\bigg(\langle c_1(\mathfrak s),c_1(\mathfrak s)\rangle - \sigma(W)\bigg) - \iota(\Sigma),
	\end{equation*}
	where $\mathfrak s$ is the spin\textsuperscript{c} structure corresponding to $z$ and $[\mathfrak a_0]$ is the reducible critical point represented by the lowest positive eigenvalue of a perturbed Dirac operator on $S^3 = \Sigma_2(S^3,\text{U}_1)$. Also, $\gr_z([\mathfrak a_0], \Sigma, [\mathfrak a])$ is the index of the perturbed Seiberg-Witten operator asymptotic to $[\mathfrak a_0]$ and $[\mathfrak b]$ in \cite[Definition~12.7]{ljk2022}.
	For reducible critical points, we modify the grading 
	\begin{equation*}
		\bar\gr^{\mathbb Q}([\mathfrak a]) =
		\begin{cases}
			\gr^{\mathbb Q}([\mathfrak a]), & [\mathfrak a] \text{ is boundary-stable},\\
			\gr^{\mathbb Q}([\mathfrak a])-1, & [\mathfrak a] \text{ is boundary-unstable}.
		\end{cases}
	\end{equation*}
\end{defn}
The number $\langle c_1(\mathfrak s),c_1(\mathfrak s)  \rangle =  c_1(\mathfrak s)^2$ is by definition the pairing $\langle c_1(\mathfrak s), \alpha \rangle$ from $H^2(W;\reals) \times H^2(W,\del W;\reals) \to \reals$, where $\alpha$ is a choice of class $H^2(W,\del W;\reals)$ that restricts to $c_1(\mathfrak s)$.
Since $\langle c_1, c_1 \rangle$, $\sigma$, and $\iota$ are additive, and the expression vanishes for closed 4-manifolds, the rational grading $\gr^{\mathbb Q}([\mathfrak a])$ does not depend on the choice of cobordism $\Sigma$ or homotopy class $z$.
Under this convention, $[\mathfrak a_0]$ on $\text{U}_1$ has zero absolute grading.

Alternatively, the absolute grading can be expressed in terms of the expected dimension of Seiberg-Witten moduli space on a 4-manifold bounding $Y$, by gluing a disc to the puncture of cobordism $\Sigma$.
Let $\hat \Sigma$ be capped surface and let $X$ be the double branched cover of $\hat \Sigma$.
\[
\gr^{\mathbb Q}([\mathfrak a]) = -\gr_z(\hat\Sigma^*,[\mathfrak a]) +\frac{1}{8}(c_1(\mathfrak s_X)^2 -\sigma(X)) -\frac{1}{2}(\chi(X) + \sigma(X) - 1 + b_1(Y_+) - b_1(Y_-)).
\]

For a link cobordism $\Sigma: K_- \to K_+$ and $\mathfrak s$ a spin\textsuperscript{c} structure on $W = \Sigma_2([0,1] \times S^3, \Sigma)$, we have the cobordism map
\begin{equation*}
	\HMR^{\circ}_*(\Sigma,\mathfrak s):
	\HMR^{\circ}_*(K_-,\mathfrak s_-) \to 
	\HMR^{\circ}_*(K_+,\mathfrak s_+), 
\end{equation*}
where $\mathfrak s_{\pm}$ is the restriction of $\mathfrak s$ on $\Sigma_2(S^3,K_{\pm}).$
If both $\mathfrak s_-$ and $\mathfrak s_+$ are torsion, then
$\HMR^{\circ}_*(\Sigma,\mathfrak s)$ has a well-defined degree
\begin{equation}
\label{eqn:deg_cob_sigma}
	\frac{1}{8}\left(c_1(\mathfrak s_W)^2  - \sigma(W)\right) - \iota(\Sigma).
\end{equation}
The degree can be rewritten as
\begin{equation*}
	\frac{1}{8}c_1(\mathfrak s_W)^2 - \frac{1}{8}
\left(\sigma(K_+) - \sigma(K_-) - \frac{\Sigma^2}{2}\right)
- \frac{1}{2}\left( \sigma(K_+) - \sigma(K_-) - \frac{\Sigma^2}{2} + b_1(\Sigma) - b_0(\Sigma) + \eta(K_+)- \eta(K_-)\right)
\end{equation*}
where $\eta(K)$ is the nullity of the link (cf. \cite{KauffmanTaylor1976}), which is equal to $b^1$ of its branched double cover.
\subsection{Fr\o yshov invariant}
\hfill \break
For a link $K \subset S^3$ with $\det(K) \ne 0$, the double branched cover is a rational homology sphere.
Given a torsion spin\textsuperscript{c} structure $\mathfrak s$ on $\Sigma_2(S^3,K)$, we have an isomorphism $\overline{\HMR}_{\bullet}(K,\mathfrak s) \cong F_2[\upsilon^{-1},\upsilon]].$
Consider the homomorphism
\begin{equation*}
	i_*:\overline{\HMR}_{\bullet}(K,\mathfrak s)
		\to 
		\widecheck{\HMR}_{\bullet}(K,\mathfrak s).
\end{equation*}
\begin{defn}
	Let $K$ be a link with nonzero determinant and $\mathfrak s$ be a torsion spin\textsuperscript{c} structure.
	The \emph{(real) Froyshov invariant} $h_R(K,\mathfrak s)$ is the number with the property that the element with lowest absolute grading in
	\begin{equation*}
		i_*(\overline{\HMR}_{\bullet}(K,\mathfrak s)) \subset \widecheck{\HMR}_{\bullet}(K,\mathfrak s_0)
	\end{equation*}
	 has $\gr^{\mathbb Q} = -h_R(K,\mathfrak s)$.
\end{defn}
We have the following monotonicity of the real Fr\o yshov invariant.
\begin{prop}
	Let $K_-, K_+$ be two links with nonzero determinants, and $\Sigma:K_- \to K_+$ be a connected cobordism.
	Let $\mathfrak s$ be a real spin\textsuperscript{c} structure on the branched cover of $\Sigma$, restricting to real spin\textsuperscript{c} structure $\mathfrak s_{\pm}$ over branched covers of $K_{\pm}$.
	Suppose the double branched cover along $\Sigma$ is negative-definite, i.e.
	\[
		b_1(\Sigma) - b_0(\Sigma)  + \sigma(K_+) - \sigma(K_-)-\frac{\Sigma \cdot \Sigma }{2}=0.
	\]
	Then
	\[
		h_R(K_-,\mathfrak s_-) \ge h_R(K_+,\mathfrak s_+) + \frac{1}{8}\left(c_1(\mathfrak s)^2  - \sigma(K_-) + \sigma(K_+) + \frac{\Sigma \cdot \Sigma}{2}\right).
	\]
\end{prop}
\begin{proof}
	The proof is unchanged from the usual Fr\o yshov inequality (see e.g. \cite{KMbook2007, FranLin2018}) for a fixed spin\textsuperscript{c} structure, which we will sketch now.
	Consider the commutating square
	\[\begin{tikzcd}
		\overline{\HMR}_{\bullet}(K_-,\mathfrak s_-)
		\ar[r] 
		\ar[d,"i_*"] 
		& 
		\overline{\HMR}_{\bullet}(K_+,\mathfrak s_+) 
		\ar[d,"i_*"]
		\\
		\widecheck{\HMR}_{\bullet}(K_-,\mathfrak s_-) 
		\ar[r] 
		& 
		\widecheck{\HMR}_{\bullet}(K_+,\mathfrak s_+) 
	\end{tikzcd}\]
	where the horizontal maps are given by the cobordism maps of $(\Sigma,\mathfrak s)$.
	The top row can be identified with $\upsilon^{-(c_1(\mathfrak s)^2 - \sigma(W))/8}$
	\[\begin{tikzcd}
	\mathbb F_2[\upsilon^{-1},\upsilon]] \ar[rrr,"\upsilon^{-(c_1(\mathfrak s)^2 - \sigma(W))/8}"]	& {} & {} & \mathbb F_2[\upsilon^{-1},\upsilon]],
	\end{tikzcd}\]
	where $W$ is the branched cover of $\Sigma$.
	Indeed, since $W$ is negative definite, there exists a unique ASD spin\textsuperscript{c} connection on the manifold $W^*$ with cylindrical ends.
	For any $j_-,j_+ \in \mathbb Z$, the compactified reducible moduli space
	\[
		M^{\red,+}_z(\mathfrak a_{j_-}, \Sigma,\mathfrak b_{j_+})
	\]
	from critical points $\mathfrak a_{j_-}$ to $\mathfrak b_{j_+}$ is either empty or a real projective space, so the resulting map is a power of $\upsilon$. 
	By Equation~\eqref{eqn:deg_cob_sigma}, the power is $-(c_1(\mathfrak s)^2 - \sigma(W))/8$.
	If $x$ is an element in 	$ \overline{\HMR}_{\bullet}(K_+,\mathfrak s_+) $ such that $i_*(x) \ne 0$ achieves the minimal degree, then 
	\[
		\upsilon^{(c_1(\mathfrak s)^2 - \sigma(W))/8} x
	\]
	is mapped under $i_*$ to an nonzero element in $\widecheck{\HMR}_{\bullet}(K_-,\mathfrak s_-) $, so 
	\[
		-h_R(K_-,\mathfrak s_-) \le - h_R(K_+, \mathfrak s_+) - \frac{1}{8}(c_1(\mathfrak s)^2 - \sigma(W)). \qedhere
	\]
\end{proof}
\begin{rem}
	One can define more generally an absolute grading for real 3-manifolds $(Y,\upiota)$ and Fr\o yshov invariant $h_R(Y,\upiota,\mathfrak s)$.
	The analogue of negative-definiteness for a real cobordism $(W,\upiota_W):(Y_-,\upiota_-) \to (Y_+,\upiota_+)$ is the condition
	\[
		b^+_{-\upiota^*}(W) = 0.
	\]
\end{rem}
\subsection{Some calculations}
\hfill \break
A priori, the real absolute grading requires a $4$-manifold bounding $\Sigma_2(S^3,K)$ equivariantly.
In some cases, the real grading can be directly deduced from the ordinary grading. 
Following \cite{LinLipnowski2022spec}, we introduce the notion of a real minimal L-space.
\begin{defn}
	A rational homology $3$-sphere with involution $(Y,\upiota)$ is a \emph{real minimal $L$-space}, if for some $\upiota$-invariant Riemannian metric, there exists no real irreducible critical point to the perturbed CSD gradient.
	Similarly, a link $K$ is \emph{minimal $L$-link}, if its branched double cover is a real minimal $L$-space.
\end{defn} 
Examples of minimal $L$-links include the 2-bridge knots and the $(p,q)$-torus knots where $p,q \ge 3$ are coprime.
The first family are branched covered by lens spaces and can be equipped with invariant positive scalar curvature metrics \cite[Section~14]{ljk2022}, from which we deduce no irreducible critical points for small perturbations.
The second family give rise to a family of Brieskorn spheres $\{\Sigma(2,p,q)\}$ as branched double covers, and one can argue directly no irreducible solutions to perturbed Seiberg-Witten equations exist for Seifert metrics \cite[Section~14.5]{ljk2022}.

Suppose $K$ is a minimal $L$-link and let $\mathfrak s$ be a real spin\textsuperscript{c} structure. Then $j_* = 0$ and the long exact sequence becomes a short exact sequence
\begin{equation*}
	\begin{tikzcd}
		0	 \ar[r] 
		& \widehat{\HMR}_{\bullet}(K,\mathfrak s) \ar[r,"p_*"]
		& \overline{\HMR}_{\bullet}(K,\mathfrak s) \ar[r,"i_*"]
		& \widecheck{\HMR}_{\bullet}(K,\mathfrak s) \ar[r]
		& 0.
	\end{tikzcd}
\end{equation*}
The following proposition applies to links whose branched cover are minimal $L$-spaces both in the real and ordinary sense.
\begin{prop}
	Suppose $\Sigma_2(S^3,K)$ is a real minimal $L$-space and ordinary minimal $L$-space, for the same choice of invariant Riemannian metric.
	Then $h_R(K,\mathfrak s) = h(\Sigma_2(K,\mathfrak s))$
\end{prop}
\begin{proof}
	We assume a small real perturbation is chosen so that there is no irreducible critical point and the spectrum of the Dirac operator is simple.
	Let $[\mathfrak b_0]$ be the boundary-stable critical point corresponding to the smallest positive eigenvalue of the perturbed Dirac operator.
	Then $[\mathfrak b_0]$ realizes the lowest rational absolute grading in $\widecheck{\HMR}_{\bullet}(K,\mathfrak s)$, given by
	\begin{equation*}
		-\gr(\hat \Sigma, [\mathfrak b_0]) + \frac{1}{8}(c_1(\mathfrak s_X)^2 - \sigma(X)) -\frac{1}{2}(\chi(X) + \sigma(X) - 1 ),
	\end{equation*}
	where $\hat \Sigma$ is a surface in $B^4$ bounding $K$ and $(X,\upiota_X)$ is the corresponding branched double cover.
	On the other hand, the expected dimension can be expressed in terms of \emph{real} indices of the linearized Seiberg-Witten operator
	\begin{equation*}
		\gr(\hat \Sigma,[\mathfrak b_0]) = 
		\ind^{-\upiota_X^*}_{L^2(X)}(d+d^*) +
		\ind_{\mathbb C, L^2(X)}(D_A^+),
	\end{equation*}
	where the first term is the $(-\upiota_X^*)$-invariant part of the $L^2$-index and the second term is the non-invariant complex index of the twisted Dirac operator.
	In particular, we do not need the equivariant index of the Dirac operator because the involutive lift of $\upiota_X$ on the spinor bundle anti-commutes with multiplication by $\sqrt{-1}$.
	Since $X$ is a cover of $B^4$, we see that
	\[
		\ind^{-\upiota_X^*}_{L^2(X,\iota_X)}(d+d^*)
		= \ind_{L^2(X)}(d+d^*) + 1
		= -\frac{1}{2}(\chi(X) + \sigma(X)+ 1) +1,
	\]
	and so
	\begin{equation}
	\label{eqn:hR1}
		-h_R(K,\mathfrak s) = - \ind_{\C,L^2(X)}(D_A^+) + \frac{1}{8} (c_1(\mathfrak s_X)^2 - \sigma(X)) 
	\end{equation}
	Since $Y = \Sigma_2(S^3,K)$ is assumed to be a minimal $L$-space itself, the lowest grading is also achieved by a critical point represented by the smallest eigenvalue of the perturbed Dirac operator (with different perturbations).
	The absolute grading is (cf. \cite[Section~39]{KMbook2007} and see \cite{LinLipnowski2022} for the definition of the Fr\o yshov invariant)
	\begin{align}
		-2h(Y,\mathfrak s) &=
		-\gr(X, [\mathfrak b_0]) + \frac{1}{4}(c_1(\mathfrak s_X)^2 - \sigma(X)) -\frac{1}{2}(\chi(X) + \sigma(X) + 1 )\\
		\label{eqn:h1}
		&= - 2\ind_{\C,L^2(X)}(D_A^+) +\frac{1}{4}(c_1(\mathfrak s_X)^2 - \sigma(X)).
	\end{align}
	From equations~\ref{eqn:hR1} and ~\ref{eqn:h1}, we obtain the equality of the Fr\o yshov inequalities.
\end{proof}
\begin{cor}
	If $K = K(p,q)$ is the $(p,q)$-rational knot branched covered by the lens space $L(p,q)$. Then $h_R(K,\mathfrak s) = h(L(p,q))$.
\end{cor}

\section{The Topology of Unoriented Skein Cobordisms}
\label{sec:top_skein}
\subsection{Skein triangle}
\hfill \break
Let us begin with the orbifold picture, described in \cite{KMunknot2011}.
Let $\mathbb Y = \mathbb Y_i = S^3$, and let $K_i \subset \mathbb Y_i$ be an unoriented skein triple of links.
The three links $\{K_2, K_1, K_0\}$ differ inside a small $3$-ball $B^3$.
There is an order-$3$ symmetry on $B^3$ containing two arcs, thought of as tetrahedra, illustrated in Figure~\ref{fig:arcsintetra}.
We extend $K_i$ to all $i \in \mathbb Z$ three-periodically.
\begin{figure}[hbt]
  \includegraphics[height=1.5in]{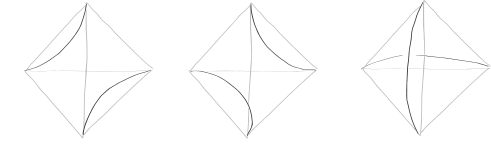}
  \caption{Arcs in tetrahedra $B^3$, with $\mathbb Z_3$-symmetry.}
  \label{fig:arcsintetra}
\end{figure}

To obtain a cobordism from the link $K_i$ to $K_{i-1}$ we glue $([0,1] \times B^3, T_{i,i-1})$ to the product surface $[0,1] \times (K_i - B^3)$, where $T_{i,i-1}$ is a rectangle bounding $B^3 \cap K_i$ and $B^3 \cap K_{i-1}$.
For arbitrary $j \le i$, we define the cobordism $\Sigma_{ij}$ by composing $\Sigma_{i,i-1}, \dots, \Sigma_{j+1,j}$.
\begin{figure}[hbt]
  \includegraphics[height=1.7in]{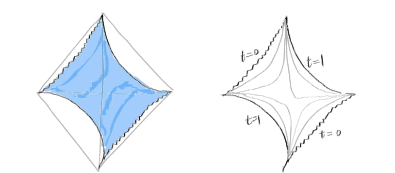}
  \caption{The rectangle $T_{i,i-1}$}
\end{figure}

We restate the main theorem as follows.
\begin{thm}
	For $\circ \in \{\vee, \wedge, -\}$, we have an isomorphism of real monopole Floer homology groups
	\begin{equation*}
		\HMR^{\circ}_{\bullet}(K_2) \cong H_{\bullet}(\text{Cone}(\HMR^{\circ}_{\bullet}(K_1) \to \HMR^{\circ}_{\bullet}(K_0))),
	\end{equation*}
	where the right hand side is the mapping cone of $\HMR^{\circ}(\Sigma_{10})$.
\end{thm}

There is an arc $\delta_i \subset B^3$ lying on $\Sigma_{i,i-1}$ at time $\frac12$, connecting the two components of the link in $B^3$, see Figure~\ref{fig:delta_arc}.
We form a M\"obius band $M_{i,i-2} \subset [0,2] \times \mathbb Y$ by taking the union of 
\begin{itemize}
	\item neighbourhoods (in $\Sigma_{i,i-2}$) of two arcs $\delta_{i+1}$ and $\delta_{i}$  at time $\frac12$ and $\frac32$, and
	\item two bands obtained from neighbourhoods of $(\delta \cap K_i) \subset K_i$, in time $[\frac12,\frac32]$.
\end{itemize}
Each M\"obius band has self-intersection $(+2)$ by \cite[Lemma~7.2]{KMunknot2011}.
\begin{figure}[hbt]
  \includegraphics[height=1.5in]{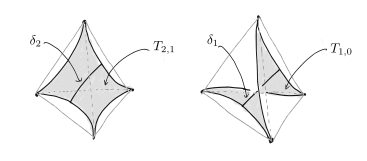}
  \caption{The arc $\delta_i \subset \Sigma_{i,i-1}$.}
  \label{fig:delta_arc}
\end{figure}
We glue the two M\"obius bands $M_{i,i-2}$ and $M_{i-1,i-3}$ to obtain $M_{i,i-3}$, which topologically is twice-punctured $\mathbb{RP}^2$.
%In particular, the branched cover of $[0,2] \times S^3$ along $M_{i,i-3}$ is 
\begin{figure}[hbt]
  \includegraphics[height=1.5in]{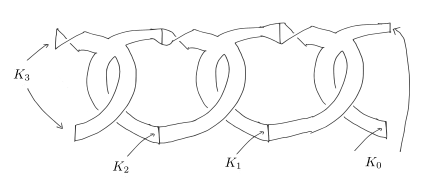}
  \caption{Gluing M\"obius bands.}
\end{figure} 

Let $\mathbb B_{i,i-2}$ be a regular neighbourhood (diffeomorphic to a $4$-ball) of $[\frac12,\frac32] \times \delta_{i,i-2}$ in $[\frac12,\frac32] \times B^3$.
Let $\mathbb B_{i,i-3}$ be the regular neighbourhood of the union $\mathbb B_{i,i-2} \cup \mathbb B_{i-1,i-3}$ which is again a $4$-ball and contains $M_{i,i-3}$.
Let $\mathbb S_{31}, \mathbb S_{20}, \mathbb S_{30}$ be the boundary $3$-spheres of $\mathbb B_{31}, \mathbb B_{20},$ and $\mathbb B_{30}$ respectively.
The sphere $\mathbb{S}_{i,i-1}$ intersects the M\"obius band $M_{i,i-1}$ at an unknot, and $\mathbb{S}_{30}$ intersects the twice-punctured $\mathbb{RP}^2$ at a 2-component unlink $\text{U}_2$. 
%We summarize our definitions as follows.
%\begin{itemize}
%	\item $\mathbb Y_i$ (3-sphere)
%	\item $\Sigma_{ij}: K_i \to K_j$ (skein cobordisms)
%	\item $\delta_{i,i-2} \subset B^3$ (arc)
%	\item $M_{31},M_{20}$ (Mobius bands)
%	\item $M_{30}$ (twice punctured $\mathbb{RP}^2)$
%	\item $\mathbb B_{31},\mathbb B_{20}$ (4-balls)
%	\item $\mathbb B_{30}$ (4-ball)
%	\item $\mathbb S_{31}, \mathbb S_{20}$ (3-spheres)
%	\item $\mathbb S_{30}$ (3-sphere)
%\end{itemize}
%\begin{figure}[hbt]
%  \includegraphics[width=2.8in]{5hypersurfaces.png}
%  \caption{}
%\end{figure}

\subsection{The double branched covers of the skein triangle}
\hfill \break
We introduce some notations for the double branched covers of what appeared in the previous subsection.
Let $Y_i$ be the branched cover of $(\mathbb Y_i,K_i)$, and $W_{ij}:{Y}_{i} \to {Y}_{j}$  be the branched cover of $([0,j-i]\times \mathbb Y,\Sigma_{ij})$.
%(The comma ``,'' between $i$ and $j$ will be omitted if there is no confusion.)
Let $U_{30}$ be the branched cover of $(\mathbb B_{30}, M_{30})$. 

Let $S_{i,i-2}$ be the double branched cover of $(\mathbb S_{i,i-2}, \mathbb S_{i,i-2} \cap M_{i,i-2})$. 
Recall that $\mathbb S_{i,i-2} \cap M_{i,i-2}$ is an unknot $\text{U}_1$ in $\mathbb S_{i,i-2}$ so $S_{i,i-2}$ is diffeomorphic to a $3$-sphere.
Let $S_{30}$ be the branched cover of $(\mathbb S_{30},\mathbb S_{30} \cap M_{30}) \cong (S^3, \text{U}_2)$, where $\text{U}_2$ is the $2$-component unlink.
Thus the hypersurface $S_{30} \subset W_{30}$ is diffeomorphic to $S^1 \times S^2$.

There are two diffeomorphisms 
\[
	U_{30} \cong  (S^2 \times D^2) \#_{S_{i,i-2}} \overline{\mathbb{CP}}^2,
\]
for $i=2,3$ corresponding to the decompositions of orbifolds
\[
	(\mathbb B_{30}, M_{30}) \cong (\mathbb B_{30},A_{i}) \# (S^4,\mathbb{RP}^2)
\] where $A_{i}$, $i=2,3$ are two annuli (cf. \cite[Section~7.3]{KMunknot2011}) within different isotopy classes of diffeomorphism.
Moreover, $\mathbb{RP}^2 \subset S^4$ has self-intersection $+2$.
The notations can be summarized in the following table.
\bgroup
\def\arraystretch{1.5}% 
\begin{center}
\begin{tabular}{ |c|c|c| } 
 \hline
  Branch Locus & Base Manifold &  Branched Cover 
 \\  \hline
 $K_i$ & $\mathbb Y_i$ & $Y_i$  
 \\ \hline
 $\Sigma_{i,j}:K_i \to K_{j}$ & $[0,i-j] \times \mathbb Y_i$ & $W_{ij}$ 
 \\ \hline
  %$\Sigma_{i,i-2}:K_i \to K_{i-2}$ & $[0,2] \times \mathbb Y_i$ & $X_{i,i-2}$ 
 %\\ \hline
  %$\Sigma_{30}:K_3 \to K_{0}$ & $[0,3] \times \mathbb Y_3$ & $V_{30}$ 
 %\\ \hline
 $M_{i,i-2}$ & $\mathbb B_{i,i-2}$ & 
 $N_{i,i-2}$
 \\ \hline
 $M_{30}$ & $\mathbb B_{30}$ & $N$
 \\ \hline
 $\mathbb S_{i,i-2} \cap M_{i,i-2}$ & $\mathbb S_{i,i-2}$ & $S_{i,i-2}$
 \\ \hline
 $\mathbb S_{30} \cap M_{30}$ & $\mathbb S_{30}$ & $S_{30}$
 \\ \hline
%  $B^3 \cap K_i$ & $B^3$ & $S^1 \times D^2$
% \\ \hline
  %See below. 
  & $[1/2,3/2] \times \delta_{i}$ & $E_{i-1}$
 \\ \hline
% cell1 & cell2 & cell3
% \\ \hline
\end{tabular}
\end{center}
\egroup

The cobordism $W_{i,i-1}$ can be seen as a surgery cobordism $Y_i \to Y_{i-1}$.
The knots on which we perform surgery are preimages in $\mathbb Y_i$ of arcs connecting the two components of $B^3 \cap K_i$.
Notice that the core of $\Sigma_2(B^3, B^3 \cap K_i) \cong S^1 \times D^2$ is the preimage of this arc. 
(A more detailed account of the surgery can be found in e.g.~\cite{OzSz2003absolute}.)

The surgery cobordisms can be described as gluing invariant $2$-handles to manifolds with torus boundaries $\Sigma_2(\mathbb Y_i - B^3, K_i - B^3)$.
A concrete model of an invariant 2-handle is $D^2 \times D^2 \subset \mathbb R^4$ equipped with the involution
\begin{equation*}
	(x_1,x_2,y_1,y_2) \mapsto (x_1,-x_2,y_1,-y_2).
\end{equation*}
%In particular, the fixed point set on the boundary of the 2-handles consists of four points $(\pm 1, 0, 0, 0)$ and $(0,0,\pm 1, 0)$.
In the base space $[0,1] \times \mathbb Y$, the quotient of the 2-handle is again a 2-handle. 
The core is $[0,1/2] \times \delta_{i+1}$ and the cocore is $[1/2, 1] \times \delta_{i}$.
In particular, a neighbourhood of the union of the cocore in $[0,1] \times \mathbb Y$ and the core in $[1,2] \times \mathbb Y$ is precisely $\mathbb B_{i,i-2}$.

 On the branched cover $W_{i,i-1}$, the cocore in $W_{i,i-1}$ and core in $W_{i-1,i-2}$ of the two 2-handles meet at the surgery knot in $Y_{i-1}$.
 The core and co-core form a 2-sphere $E_{i-1}$ of self-intersection $(-1)$. The boundary of a regular neighbourhood of $E_i$ is $S_{i,i-2}$.

\section{Proof of the Unoriented Skein Exact Triangle}
\label{sec:prof_skein}
Our proof of the exact triangle is modelled on~\cite{KMOS2007} but we will follow the orbifold notations in \cite{KMunknot2011}.
In this section, we will prove the theorem for the ``from'' version.
The other two flavours can be treated similarly.
\subsection{Homological algebra}
\hfill \break
We need the following lemma in homological algebra (cf. \cite[Lemma~4.2]{OzSz2005HFdc}).
\begin{lem}[Triangle Detection]
Suppose for each $i \in \mathbb Z$ we have a chain-complex $(C_i,d_i)$ over $\mathbb F_2$ and chain maps
\begin{equation*}
	f_i: C_i \to C_{i-1}.
\end{equation*}
Suppose that the composite chain map $f_{i-1} \circ f_i$ is homotopic to $0$ via a chain homotopy $H_i$, such that
\begin{equation}
	d_{i-1}H_i + H_id_i + f_{i-1}f_i = 0,
\end{equation}
for all $i$.
Moreover, suppose for all $i$ the map
\begin{equation*}
	H_{i-1}f_i + f_{i-2} H_i : C_i \to C_{i-3},
\end{equation*}
induces an isomorphism in homology.
Then the induced maps $(f_i)_*$ in homology
\begin{equation*}
	(f_i)_*:H_*(C_i,d_i) \to H_*(C_{i-1},d_{i-1})
\end{equation*}
form an exact sequence.
For each $i$ the anti-chain map
\begin{equation*}
	\Phi: C_i \ni s \mapsto (f_is,H_is) \in \text{Cone}(f_i).
\end{equation*}
induces isomorphism in homology.
\end{lem}
To apply the triangle detection lemma, we will make the following substitution.
\begin{itemize}[leftmargin=*]
	\item Let  $C_i = \check C(K_i)$ be the real monopole Floer chain group (over all real spin\textsuperscript{c} structures).
	\item Let $\check \del: C_i \to C_i$ be the ``from'' differential, where $i$ is not to confused with the grading of the Floer homology.
	\item Let $F_{i}:C_i \to C_{i-1}$ be the cobordism map
	\[
		\check m(\Sigma_{i,i-1}) = 
		\sum_{(\mathfrak s, \uptau)} 
		\check m(\Sigma_{i,i-1};\mathfrak s, \uptau).
	\] 
	\item $H_{i}:C_i \to C_{i-2}$ will be a chain homotopy map, defined in Subsection~\ref{subsec:firsthomotopy}.
	\item $G_{i}:C_i \to C_{i-3}$ will be another chain homotopy, defined in Subsection~\ref{subsec:secondhomotopy}.
\end{itemize}
The homotopies should satisfy the relations
\begin{equation}
	\check\del H_{20} + H_{20}\check\del + F_{10}F_{21} = 0
\end{equation}
and
\begin{equation}
\label{eqn:triangle_detect_iso}
	\check{\del}G_{30} + G_{30}\check\del + F_{10}H_{31} + H_{20}F_{32} + \check L = 0,
\end{equation} 
where $\check L$ is a chain map that is homotopic to the identity.
Each of the chain homotopy will arise from counting moduli spaces parameterized by families of metrics.

%\newpage
\subsection{The first homotopy $H$}\label{subsec:firsthomotopy}
\hfill \break
We consider the double composition $\Sigma_{20}:K_2 \to K_0$, as the construction for general $\Sigma_{i,i-2}$ is identical.
There is a decomposition of  $\Sigma_{20}$ along $\del M_{20}$
\begin{equation*}
	([0,2] \times \mathbb Y, \Sigma_{20}) =
	([0,2] \times \mathbb Y - \mathbb{B}_{20}, \Sigma_{20} - M_{20})
	\cup_{(\del \mathbb{B}_{20}, \del M_{20} )} ( \mathbb{B}_{20}, M_{20} ),
\end{equation*}
Let $V_{20}: K_2 \to K_0$ be a cobordism obtained from gluing a standard disc back to $\Sigma_{20} - M_{20}$.
Then $\Sigma_{20}$ is a direct sum of $V_{20}$ with a $\mathbb{RP}^2 \subset S^4$, where $\mathbb{RP}^2$ has self-intersection $(+2)$:
\[
	([0,2] \times \mathbb Y, \Sigma_{20}) = ([0,2] \times \mathbb Y, V_{20}) \# (S^4, \mathbb{RP}^2)
\]
There is another decomposition  along $(\mathbb Y_i, K_i)$:
\begin{equation*}
	(\mathbb Y \times [0,2], \Sigma_{20}) = 
	(\mathbb Y \times [0,1], \Sigma_{21}) \cup_{(\mathbb Y_1, K_1) }
	(\mathbb Y \times [1,2], \Sigma_{10})
\end{equation*}
The hypersurfaces $\mathbb Y_1$ and $\mathbb S_{20}$ intersect at a 2-sphere, where the singular loci intersect at four points on the 2-sphere.

For the double branched cover $W_{20}$,
there are corresponding decompositions
\begin{equation*}
	W_{20} = (W_{20} - N_{20}) \cup_{S_{20}} N_{20}
    = W_{20} = W_{21} \cup_{Y_1} W_{10}.
\end{equation*}
Diffeomorphically, $N_{20} \cong \overline{\mathbb{CP}}^2$ is the branched cover of $S^4$ along a $\mathbb{RP}^2$ with self-intersection $+2$.
The intersection $Y_1 \cap S_{20}$ is a 2-torus.

\begin{figure}[hbt]
  \includegraphics[height=2in]{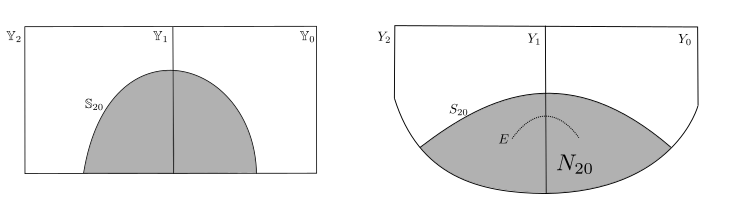}
  \caption{Hypersurfaces in the double composite cobordism.}
\end{figure}
Let $Q(\mathbb{S}_{20},\mathbb{Y}_1)$ be the family of orbifold metrics obtained from inserting a cylinder $[T,-T] \times \mathbb S_{20}$ when $T$ is negative, and a cylinder $[-T,T] \times \mathbb Y_1$ when $T$ is positive.
Extend the metric to $\bar Q = \bar Q(\mathbb S_{20},\mathbb{Y}_1) \cong [-\infty,\infty]$ by adding broken orbifold metric, along $\mathbb{S}_{20}$ at $T = -\infty$ and along $\mathbb{Y}_1$ at $T = +\infty$.
We pull back the orbifolds metrics above to
obtain $\bar Q$-family of invariant smooth Riemannian metrics on the branched cover $W_{20}$.

We can arrange the 1-parameter family of  metrics on $W_{20}$ so that it has positive scalar curvature over $S_{20}$ and is product-like near the torus $Y_1 \cap S_{20}$.
Furthermore, we can arrange the metrics so that at time $T = -\infty$, the punctured $\overline{\mathbb{CP}}^2$ component has positive scalar curvature also.
This follows from \cite[Section~1]{GromovLawsonPSC} or using the explicit PSC metrics in \cite{GongThesis2018}.
(We put no restriction on the scalar curvature for parameter $T$ small, except near the neighbourhoods of hypersurfaces.)

Consider the real Seiberg-Witten moduli space over the family $\bar Q$
\begin{equation*}
	M_z(\mathfrak a; \Sigma_{20}^*; \mathfrak b)
	\to
	\bar Q,
\end{equation*}
and its compactification $M^+_z(\mathfrak a,\Sigma_{20}^* ,\mathfrak b)_{\bar Q}$ by adding broken trajectories.
Denote the corresponding reducible Seiberg-Witten moduli space $M^{\red}_z(\mathfrak a, \Sigma_{20}^*,\mathfrak b)_{\bar Q}$.

Counting the zero dimensional components, we obtain the maps $H^o_o, H^u_o, H^u_s, \bar H^s_s, \bar H^s_u, \bar H^u_s,$ and $\bar H^u_u$, where for example
\begin{equation*}
	H^o_u: C^o(K_2) \to C^u(K_0).
\end{equation*}
The homotopy $H_{20}$ is defined in the following matrix form with respect to the decomposition $C_i = C_i^o \oplus C_i^s$.
\begin{equation*}
	H_{20} =
	\begin{pmatrix}
	H^o_o & H^u_o \bar\del^s_u + m^u_o(\Sigma_{10})\bar m^s_u(\Sigma_{21}) + \del^u_0 \bar H^s_u\\
	H^o_s & \bar H^s_s + H^u_s \bar\del^s_u + m^s_u(\Sigma_{10})\bar m_u^s(\Sigma_{21}) + \del^u_s \bar H^s_u	
	\end{pmatrix}.
\end{equation*}

\begin{prop}\label{prop:Hidentity}
	For suitable small perturbation on $(\mathbb S_{20},\mathbb S_{20} \cap M_{20})$, we have
	\begin{equation*}
		\check\del \check H_{20} + H_{20}\check\del + F_{10}F_{21} = 0.
	\end{equation*}	
%(Recall $\mathbb S_{20} \cap M_{20}$ is an unknot.)
\end{prop}

\begin{proof}
The argument is the same as \cite[Proposition~5.2]{KMOS2007},  while the index formulae are different.
The idea is to stretch the necks along hypersurfaces and examine the boundaries of 1-dimensional moduli spaces.
In addition to the boundary points that appeared in the proof of composition law
\cite[Proposition~13.3]{ljk2022}, there are contributions from the fibre of $M_z^+(\mathfrak a, W_{20}^*,\mathfrak b)_{\bar Q}$ over $T=-\infty$.

A typical element of the $T=-\infty$ fibre is of the form
	\begin{equation*}
		(
		\check\upgamma_{K_2},
		\check\upgamma_{\mathbb{S}_{20}},
		\check\upgamma_{K_0},
		\upgamma_{V_{20}},
		\upgamma_{M_{20}})
	\end{equation*}
where the first and the third are real trajectories on $\Sigma_2(\mathbb Y_2, K_2)$ and $\Sigma_2(\mathbb Y_0, K_0)$.
The elements $\upgamma_{V_{20}}$ and $\upgamma_{M_{20}}$ are solutions on the branched covers along $(V_{20} - \Delta)$ and $M_{20}$ with cylindrical ends attached.
Lastly, the $\check\upgamma_{\mathbb{S}_{20}}$ is a trajectory on the branched cover of $(S^3,\text{U}_1)$.
Under a small perturbation as in Example~\ref{exmp:U1}, there is a single tower $\{\mathfrak a_i\}$ of critical points, and there are zero trajectories modulo two between any pair of critical points.

The key step of the proof is to show $\upgamma_{M_{20}}$ comes in pairs.
To this end, let $z_k$ be the component of moduli space corresponding to the real spin\textsuperscript{c} structure $\mathfrak t_k$ on 
$$
N_{20} = \Sigma_2(\mathbb B_{20}, M_{20}) \cong \overline{\mathbb{CP}}^2 \setminus \Ball
$$ 
with $\langle c_1(\mathfrak t_k),[E]\rangle = 2k-1$ for the $(-1)$-self-intersection sphere $E$.
For a sufficiently small perturbation on $\Sigma_2(\mathbb B_{20},M_{20})$, we have the following.
\begin{itemize}[leftmargin=*]
		\item The formal dimension of $M_{z_k}(M_{20},\mathfrak a_i)$ is
		\[
		\gr_{z_k}(M_{20},\mathfrak a_i) =  \begin{cases}
			 -k(k-1)/2 - i & i \ge 0\\
			 -k(k-1)/2 - i - 1 & i < 0
		\end{cases}.\]
		\item The moduli spaces $M_z(M_{20},\mathfrak a_{i})$ contain no irreducibles and are empty for $i \ge 0$.
		\item For $i < 0$, the moduli space $M_{z_k}(M_{20},\mathfrak a_i)$ consists of a single point when it has dimension zero.
		\item The moduli spaces asscoiated to conjugate real spin\textsuperscript{c} structures $M_{z_k}(M_{20},\mathfrak a_i)$ and $M_{z_{1-k}}(M_{20},\mathfrak a_i)$ are isomorphic.
	\end{itemize}
The first bullet point uses the dimension formulae in \cite[Proposition~5.2]{KMOS2007}.
The second bullet point follows from the first, except in the case $i = 0$ and $k \in \{0,-1\}$.
In these two exceptional cases, the formal dimension is zero 
but the moduli space is empty by the positive scalar curvature assumption on $N_{20}$.
The last bullet point is a consequence of the fact the deck transformation on $N_{20}$ interchanges $z_k$ and $z_{-1-k}$.
The formulae regarding the boundary contributions can be proved the same way as \cite[Lemma~5.3]{KMOS2007}.
\end{proof}

\subsection{The second homotopy $G$}\label{subsec:secondhomotopy}
\hfill \break
The construction of $G$ involves a 2-dimensional family of metrics over $([0,3] \times \mathbb Y, \Sigma_{30})$.
\subsection*{Hypersurfaces and pentagon of metrics}
\hfill \break
Inside the triple composition $([0,3] \times \mathbb Y, \Sigma_{30})$, there are five hypersurfaces 
\begin{equation*}
	\mathbb Y_2, \mathbb S_{30}, \mathbb Y_1,  \mathbb S_{20}, \mathbb S_{31},
\end{equation*}
where two of the members intersect only if they are adjacent in the above order (cyclically). See Figure~\ref{fig:5hypersurfaces}.
\begin{figure}[hbt]
	\includegraphics[height=1.5in]{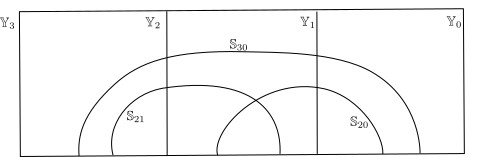}
	\caption{Five hypersurfaces in the triple composite cobordism.}
	\label{fig:5hypersurfaces}
\end{figure}
For any pair $(\mathbb S, \mathbb S')$ of disjoint hypersurfaces, we define a $[0,\infty)^2$-family of metrics $P(\mathbb S,\mathbb S')$ by stretching along the pair by $(T,T') \in [0,\infty)^2$ as in the construction of $H$.
This family of metrics extends naturally to $\bar P(\mathbb S,\mathbb S')$ by adding broken metrics.

There are five disjoint pairs and the five rectangles of Riemannian metrics are glued along their common edges to a pentagon $\bar P$.
The boundary of $\bar Q$ consists of broken metrics which we denote
\begin{equation*}
	\del \bar P = \bar Q(\mathbb Y_2) \cup \bar Q(\mathbb S_{30}) \cup \bar Q(\mathbb Y_1) \cup \bar Q(\mathbb S_{20}) \cup \bar Q(\mathbb S_{31}).
\end{equation*}
\begin{figure}[hbt]
  \includegraphics[width=3in]{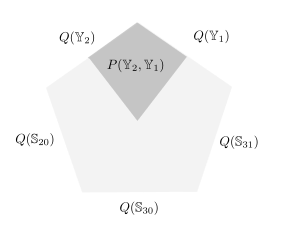}
  \caption{The pentagon $\bar P$ of metrics.}
\end{figure}
We pull back the metrics to obtain a family of invariant metrics on $W_{30}$, stretched along the hypersurfaces
\begin{equation*}
	Y_2, S_{30}, Y_1, S_{20}, S_{31}.
\end{equation*}

Let $\mathfrak a, \mathfrak b \in \mathfrak C(K_0)$.
For the zero-dimensional moduli space $M_z^+(\mathfrak a, \Sigma_{30}^*,\mathfrak b)$ and reducible $M_z^{\red,+}(\mathfrak a, \Sigma_{30}^*,\mathfrak b)$, we define the mod-2 counts
\begin{equation*}
	m_z(\mathfrak a, \Sigma_{30}^*, \mathfrak b), \quad
	\bar m_z(\mathfrak a, \Sigma_{30}^*, \mathfrak b),
\end{equation*}
respectively.
By varying the type of critical points (i.e. interior, boundary-stable and boundary-unstable), $m_z$ and $\bar m_z$ will define the matrix entries $\{G^o_o, G^o_s, G^u_0, \bar G^s_s, \bar G^s_u, \bar G^u_s, \bar G^u_u\}$.
For instance, $G^o_o$ is a map
\begin{equation*}
	G^o_o: C^o_{\bullet}(K_3) \to C^o_{\bullet}(K_0).
\end{equation*}

To reveal identity~\eqref{eqn:triangle_detect_iso}, we also consider the 1-dimensional moduli spaces $M_z^+(\mathfrak a, \Sigma_{30}^*,\mathfrak b)$.
Let $A^o_o$ be the mod-2 count of this moduli space, when $\mathfrak a$ and $\mathfrak b$ are both interior.
Then
\begin{equation*}
	A^o_o = 0.
\end{equation*}
%We analyze the contributions to $A^o_o$ as follows.
\subsection*{Contributions to $A^o_o$}
\subsection*{(a)}
There are endpoints of $M_z^+(\mathfrak a, \Sigma_{30}^*,\mathfrak b)$ that lie above the interior points $P \subset \bar P$.
They arise from the strata:
\begin{itemize}
	\item $\check M_{z_1}(\mathfrak a,K_3, \mathfrak a_1) \times M_{z_2}(\mathfrak a_1, \Sigma_{30}^*,\mathfrak b)_{P}$,
	\item $M_{z_1}(\mathfrak a, \Sigma_{30}^*, \mathfrak b_1)_P \times \check M_{z_2}(\mathfrak b_1,K_0,\mathfrak b)$,
	\item $\check M_{z_1}(\mathfrak a, K_3, \mathfrak a_1) \times \check M_{z_2}(\mathfrak a_1,K_3,\mathfrak a_2) \times M_{z_3}(\mathfrak a_2,\Sigma_{30}^*,\mathfrak b)_P$,
	\item $\check M_{z_1}(\mathfrak a, K_3,\mathfrak a_1) \times M_{z_2}(\mathfrak a_1,\Sigma_{30}^*,\mathfrak b_1)_P  \times \check M_{z_3}(\mathfrak b_1,K_0,\mathfrak b)$,
	\item $M_{z_1}(\mathfrak a,\Sigma_{30}^*,\mathfrak b_1)_P \times \check M_{z_2}(\mathfrak b_1, K_0, \mathfrak b_2) \times \check M_{z_3}(\mathfrak b_2,K_0,\mathfrak b)$,
\end{itemize}
where the middle terms in the triple product  are boundary-obstructed.
Together, they contribute
\begin{equation*}
	G^o_o \del^o_o + \del^o_oG^o_o + \del^u_o\bar\del^s_uG^o_s + \del^u_o\bar G_u^s \del^o_s + G^u_o\delbar^s_u\del^o_s
\end{equation*}
to $A^o_o$.

\subsection*{(b)}
The contributions from the edges $\bar Q(\mathbb S_{31})$ and $\bar Q(\mathbb S_{20})$ in the boundary $\del \bar P$ are zero.
Indeed, if the metric is broken along one $\mathbb S_{31}$ and $\mathbb S_{20}$, the cobordism $\Sigma_{30}$ splits off a standard $\mathbb{RP}^2$ inside $ S^4$ as connected summand, having self-intersection $+2$.
The argument in Proposition~\ref{prop:Hidentity} shows that the boundary points are even.

\subsection*{(c)}
The edge $\bar Q(\mathbb Y_1)$ correspond to the decomposition
\begin{equation*}
	([0,3] \times \mathbb Y,\Sigma_{30}) =([0,2] \times \mathbb Y, \Sigma_{31}) \cup_{\mathbb Y_1} ([0,1] \times \mathbb Y, \Sigma_{10}),
\end{equation*}
where the metric is being streched along $\mathbb Y_2$ and $\mathbb S_{31}$, and is constant on $([0,1] \times \mathbb Y, \Sigma_{10})$.
The endpoints are either from a product of two factors of form
\begin{equation*}
	M_{z_1}(\mathfrak a, \Sigma_{31}^*, \mathfrak a_1)_{\bar Q} \times 
	M_{z_2}(\mathfrak a_1, \Sigma_{10}^*, \mathfrak b),
\end{equation*}
or product of three factors like \textbf{(a)}, containing a boundary obstructed factor.
The contribution is
\begin{equation*}
	m^o_o(\Sigma_{10})H^o_o(\Sigma_{31}) + 
	m^u_o(\Sigma_{10})\bar H^s_u(\Sigma_{31})\del^o_s + m^u_o(\Sigma_{10})\delbar^s_u H^o_s(\Sigma_{31})
	+ \del^u_o\bar m^s_u(\Sigma_{10})H^o_s(\Sigma_{31}).
\end{equation*}

\subsection*{(d)}
The contribution from $\bar Q(\mathbb Y_2)$ is similar:
\begin{equation*}
	H^o_o(\Sigma_{20})m^o_o(\Sigma_{32}) + 
	H^s_u(\Sigma_{20})\bar m^u_o(\Sigma_{32}) \del^o_s +
	H^u_o(\Sigma_{20})\delbar^s_u m^o_s(\Sigma_{32})+ 
	\del^u_o\bar H^s_u(\Sigma_{20})m^o_s(\Sigma_{32}).
\end{equation*} 

\subsection*{(e)}
There is a contribution at vertex where $\bar Q_{\mathbb Y_1}$ and $\bar Q_{\mathbb Y_2}$ meet, which comes from the moduli space
\begin{equation*}
	M_{z_1}(\mathfrak a, \Sigma_{32}^*, \mathfrak a_1) \times
	M_{z_2}(\mathfrak a_1, \Sigma_{21}^*, \mathfrak a_2) \times
	M_{z_3}(\mathfrak a_2, \Sigma_{10}^*,\mathfrak b),
\end{equation*}
where the middle term is boundary is boundary-obstructed.
The contribution is
\begin{equation*}
	m^u_o(\Sigma_{10})\bar m^s_u(\Sigma_{21}) m^o_s(\Sigma_{32})
\end{equation*}
to $A^o_o$.

\subsection*{(f)}
The last edge $\bar Q(\mathbb S_{30})$ is the most interesting one.
Geometrically we have the decomposition
\begin{equation*}
	([0,3] \times \mathbb Y, \Sigma_{30}) = ([0,3] \times \mathbb Y - B^4, [0,3] \times K_0 - \Delta) \cup_{(\mathbb S_{30}, \del  M_{30} )} (\mathbb B_{30}, M_{30}).
\end{equation*}
Here $\Delta$ is the union of two standard 2-discs in a 4-ball $B^4$, and recall $\del M_{30} \cong \text{U}_2$ consists of two unknotted unlinked circles.
In other words, if one cuts off $(\mathbb B_{30}, M_{30})$  from $([0,3] \times \mathbb Y, \Sigma_{30})$, then the resulting orbifold is simply the product cobordism with two standard discs in a 4-ball removed.
Let
\[
	\mathbb U  = ([0,3] \times \mathbb Y - B^4, [0,3] \times K_0 - \Delta).
\]
The orbifold $\mathbb U^*$ has three cylindrical ends 
\begin{equation*}
	(\mathbb Y_3, K_3) \sqcup (\mathbb Y_0, K_0) \sqcup
	(\mathbb S_{30}, \del M_{30})
\end{equation*}
and can be viewed as a cobordism
$
	\mathbb U^*: (\mathbb Y_3, K_3) \sqcup (\mathbb S_{30}, \del M_{30}) \to (\mathbb Y_0, K_0).
$
The second summand $(B_{30}, M_{30})$ is a 4-orbifold with oriented boundary $(\mathbb S_{30}, \del M_{30})$.
The family of metrics $\bar Q(\mathbb S_{30}) \simeq [-\infty, +\infty]$ is constant on $\mathbb U^*$.
On the other hand, for $T \in \bar Q = \bar Q(\mathbb S_{30})$, the metric on $\mathbb B_{30}$ is stretched along $\mathbb S_{31}$ for $T < 0$ and $\mathbb S_{20}$ for $T > 0$ on $(\mathbb B_{30}, M_{30})$.

Choose a sufficiently small perturbation on $\Sigma_2(\mathbb S_{30}, \del M_{30})$, so that there are no irreducible critical points, and no irreducible trajectories on $\reals \times (\mathbb S_{30}, \del M_{30})$.
Let $\mathfrak a'$ be a critical point of the perturbed real Seiberg-Witten equation.
We use the zero-dimensional parametrized moduli spaces $M_z(\mathbb B_{30}^*, M_{30}^*, \mathfrak a')_{\bar Q}$ and $M_z^{\red}(\mathbb B_{30}^*, M_{30}^*,\mathfrak a')_{\bar Q}$ to define the elements
\begin{align*}
	n_s &\in C^s_{\bullet}(\del M_{30}), \\
	n_o &\in C^o_{\bullet}(\del M_{30}), \\
	\bar n_s &\in C^s_{\bullet}(\del M_{30}), \\
	\bar n_u &\in C^u_{\bullet}(\del M_{30}).
\end{align*}
We will prove in Corollary~\ref{cor:M30_dim} that
	$\bar n_s$ is zero, and for now we take it as granted.

Let $\mathfrak a$ be a critical point on $(\mathbb Y_3, K_3)$, let
$\mathfrak a'$ be a critical point on $(\mathbb S_{30}, \del M_{30})$, and $\mathfrak b$ be a critical point on $(\mathbb Y_0, K_0)$.
We have the moduli spaces
\begin{equation*}
	M_z(\mathfrak a', \mathfrak a, \mathbb U^*, \mathfrak b) \text{ and }
	M_z^{\red}(\mathfrak a', \mathfrak a, \mathbb U^*, \mathfrak b).
\end{equation*}
We count the zero dimensional moduli spaces above to define the matrix entries of maps
\begin{equation*}
	m^{uo}_o: C^u_{\bullet}(\mathbb S_{30}, \del M_{30}) \otimes C^o_{\bullet}(K_3)
	\to C^o_{\bullet}(K_3),
\end{equation*}
and similarly
\begin{equation*}
	m^{uu}_o, m^{uo}_s, m^{uu}_s,
	\bar m^{ss}_s, \bar m^{ss}_u, \bar m^{su}_u,
	\bar m^{su}_u, \bar m^{us}_s, \bar m^{us}_u,
	\bar m^{uu}_s, \bar m^{uu}_u.
\end{equation*}
In particular, the maps $\bar m^{ss}_s$, $\bar m^{su}_u$, and $\bar m^{us}_u$ arise from boundary-obstructed moduli spaces.
The moduli space $M_z(\mathfrak a',\mathfrak a, \mathbb U^*,\mathfrak b)$ contributing to $\bar m^{ss}_u$ are  boundary-obstructed with corank-2; that is, the formal dimension of the moduli spaces are $\gr_z(\mathfrak a', \mathfrak a, \mathbb U^*, \mathfrak b) = -2$.

The end points belonging to $\bar Q_{\mathbf S_{30}}$ of the 1-dimensional moduli space $M^+_z(\mathfrak a, \Sigma_{30}^*, \mathfrak b)_{\bar P}$ 
%(to $A^o_o$) 
	come in four sorts.
\begin{itemize}
%[leftmargin=*]
	\item two factors:
	\[M_{z_1}(B_{30}^*, M_{30}^*,\mathfrak a') \times M_{z_2}(\mathfrak a',\mathfrak a, \mathbb U^*, \mathfrak b),\]
	where $\mathfrak a'$ is necessarily boundary-unstable;
	\item three factors: such cannot exists because $\mathfrak a$ and $\mathfrak b$ are irreducible;
	\item four factors: one of the factor must be doubly-boundary-obstructed
	\[
	M_{z_1}(B_{30}^*, M_{30}^*,\mathfrak a')_{\bar Q} \times
	M_{z_2}(\mathfrak a, \mathfrak a_1) \times
	M_{z_3}(\mathfrak a', \mathfrak a_1, \mathbb U^*, \mathfrak b_1) \times
	M_{z_4}(\mathfrak b_1,\mathfrak b),
	\]
	where $\mathfrak a'$ is boundary-stable, $\mathfrak a_1$ is boundary-stable, $\mathfrak b_1$-unstable.
\end{itemize}
We get terms
\begin{equation*}
	m^{uo}_o(\bar n_u \otimes \cdot) + \del^u_o \bar m^{ss}_u(n_s \otimes \del^o_s(\cdot)).
\end{equation*}

\subsection*{(g)}
Finally, we sum over over all contributions of boundary and interior contributions of end points of the 1-dimensional moduli space $M^+_z(\mathfrak a, \Sigma_{30}^*, \mathfrak b)_{\bar P}$ from \textbf{(a)} to \textbf{(f)} discussed thus far:
\begin{align*}
	A^{oo}_o 
	&= G^o_o \del^o_o + \del^o_oG^o_o + \del^u_o\bar\del^s_uG^o_s + \del^u_o\bar G_u^s \del^o_s + G^u_o\delbar^s_u\del^o_s\\
	&+ m^o_o(\Sigma_{10})H^o_o(\Sigma_{31}) + 
	m^u_o(\Sigma_{10})\bar H^s_u(\Sigma_{31})\del^o_s \\
	&+ m^u_o(\Sigma_{10})\delbar^s_u H^o_s(\Sigma_{31})
	+ \del^u_o\bar m^s_u(\Sigma_{10})H^o_s(\Sigma_{31})\\
	&+ H^o_o(\Sigma_{20})m^o_o(\Sigma_{32}) + 
	H^s_u(\Sigma_{20})\bar m^u_o(\Sigma_{32}) \del^o_s\\
	&+
	H^u_o(\Sigma_{20})\delbar^s_u m^o_s(\Sigma_{32})+ 
	\del^u_o\bar H^s_u(\Sigma_{20})m^o_s(\Sigma_{32})\\
	&+m^u_o(\Sigma_{10})\bar m^s_u(\Sigma_{21}) m^o_s(\Sigma_{32})\\
	&+ m^{uo}_o(\bar n_u \otimes \cdot) + \del^u_o \bar m^{ss}_u(n_s \otimes \del^o_s(\cdot))\\
	&= 0.
\end{align*}

\subsection*{The rest of end-points} 
\hfill \break
As illustrated above for $A^{o}_o$, the computations in the ordinary setting applies to the real setting.
We refer the reader to \cite[Section~5]{KMOS2007} for formulae of $A^{**}_*$ and $\bar A^{**}_*$.

\subsection{The key identity}
The goal of this subsection is to verify the the identity~\eqref{eqn:triangle_detect_iso} and the quasi-isomorphism hypothesis of $\check L$.
We
define the map
\begin{equation*}
	\check L: \check C_{\bullet}(K_3) \to \check C_{\bullet}(K_3) = \check C_{\bullet}(K_0)
\end{equation*}
by
\begin{equation*}
	\check L = \begin{bmatrix}
		L^o_o && L^u_o\bar\del^s_u + \del^u_o\bar L^s_u\\
		L^o_s && \bar L^s_s + L^u_s\bar\del^s_u + \del^u_s \bar L^s_u
	\end{bmatrix},
\end{equation*}
where we insert $\bar n_u$ into the $(\mathbb S_{30}, \del M_{30})$ end of the cobordism $\mathbb U$, i.e. 
\begin{equation*}
	L^o_o = m^{uo}_o(\bar n_u \otimes \cdot),
\end{equation*}
and the rest of the entries are defined analogously.
Moreover, we define $\check G: \check C_{\bullet}(K_3) \to \check C_{\bullet}(K_3)$ by the formula
\begin{equation*}
	\check G = \begin{bmatrix}
		a & b\\
		c & d
	\end{bmatrix},
\end{equation*}
where
\begin{align*}
	a &= G^o_o\\
	b &= \del^u_o \bar G^s_u + G^u_o \delbar^s_u + m^u_o \bar H^s_u + H^u_s \bar m^s_u + \del^u_o(\bar m_u^{ss}(n_s \otimes \cdot))\\
	c &= G^o_s\\
	d &= \bar G^s_s + \del^u_s \bar G^s_u +
	G^u_s\delbar^s_u + m^u_s\bar H^s_u + H^u_s \bar m^s_u + \del^u_s\bar m^{ss}_u(n_s \otimes \cdot) + \bar m^{ss}_s(n_s \otimes \cdot),
\end{align*}
and where for example we have written $m^u_s\bar H^s_u$ as an abbreviation for $m^u_o(\Sigma_{10})\bar H^s_u(\Sigma_{31})$.

\begin{prop}
We have the identity:
\begin{equation}
\label{eqn:gauge_triangle_isom}
	\check\del \check G + \check G \check \del
	=
	\check m(\Sigma_{10})\check H_{31} + \check H_{20}\check m(\Sigma_{32}) + \check L,
\end{equation}
where $\check H_{i,i-2}$ is the homotopy map defined for the double composition $W_{i,i-2}$.
\end{prop}
\begin{proof}
	The proof is the same as \cite[Proposition~5.5]{KMOS2007}, from ends of moduli spaces over
	\begin{equation*}
		\mathbb U^*, (\reals \times \mathbb Y, \Sigma_{i,i-2}^*) , (\reals \times \mathbb Y, \Sigma_{i,i-1}^*) \text{, and } (\reals \times \mathbb Y_i, \reals \times K_i).\qedhere
	\end{equation*}
\end{proof}
The final step is the following proposition, and the proof will occupy the succeeding subsection.
\begin{prop}\label{prop:Lisom}
	The map $\check L$ induces isomorphisms in homology.
\end{prop}

\subsection{Proof of Proposition~\ref{prop:Lisom}}
\hfill \break
Let us take a closer look at the double branched cover picture in Part~\textbf{(f)}.
The cobordism $W_{30}$ decomposes into two parts
\begin{equation*}
	W_{30} = U \cup_{S_{30}} N,
\end{equation*}
where $U$ and $N$ are branched covers of $\mathbb U$ and $(\mathbb B_{30}, M_{30})$, respectively, and  $S_{30} = \Sigma_2(\mathbb S_{30}, \text{U}_2)$ is diffeomorphic to $ S^1 \times S^2$.

The manifold $N$ contains two 2-spheres $E_2$ and $E_3$ of self-intersection $(-1)$ and $E_2 \cdot E_3 = 1$.
Recall there are two ways of representing $N$ as blow-ups 
\[
	N \cong (S^2 \times D^2) \#_{S_{i,i-2}} \overline{\mathbb{CP}}^2	
\]
for $i = 2$ or $3$. 
The sphere $S_{i,i-2}$ can be thought of as the boundary of a regular neighbourhood of $E_{i}$.

The $\bar Q$-family of invariant Riemannian metric is constant on $U$. 
On the $N$ part, $\bar Q$ stretches the neck along the sphere $S_{31}$ for $[-\infty,0)$ and along  $S_{20}$ for $(0,+\infty]$.
We arrange the metric so that it is cylindrical and has positive scalar curvature near the boundary $S_{30}$, and at $T=\pm \infty$ has positive scalar curvature over the punctured $\overline{\mathbb{CP}}^2$-components.

A spin\textsuperscript{c} structure is uniquely determined by the first Chern class evaluated on the divisor $E_2$.
We denote $\mathfrak t_k$ the spin\textsuperscript{c} structure with the property
\begin{equation*}
	\langle c_1(\mathfrak t_k), E_2\rangle = 2k + 1.
\end{equation*}
Since $N$ is the branched cover of the 4-ball $\mathbb B_{30}$, any spin\textsuperscript{c} structure admits a unique compatible real structure.
We continue to suppress the choice of real structures in notations.
The conjugate real spin\textsuperscript{c} structure of $\mathfrak t_k$ by this convention is $\mathfrak t_{-1-k}$.

\begin{figure}[hbt]
  \includegraphics[width=3in]{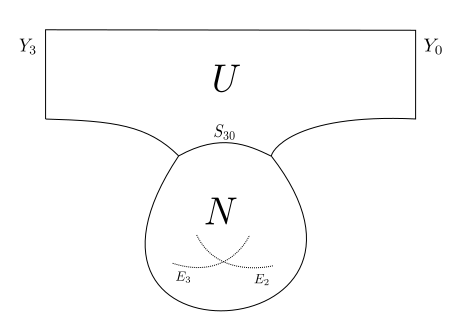}
  \caption{Decomposition of $W_{30}$ along $S_{30}$.}
\end{figure}

Recall the chain complexes of $\HMR^{\circ}(\text{U}_2)$ with a positive scalar curvature metric on $S^1 \times S^2$ and a small regular perturbation in Example~\ref{exmp:U2}: there are two towers of critical points
\begin{equation*}
	\{\mathfrak a^1_i\}, \{\mathfrak a^0_i\}: i \in \mathbb Z,
\end{equation*}
and all differentials are zero.

The next lemma follows from \cite[Lemma~5.7]{KMOS2007}, where the factors of $(1/2)$ arise from taking the real parts of the indices of Dirac operators.
\begin{lem}
\label{lem:M30_dim_stable}
	The dimension of the real moduli space $M_k(M_{30}^*,\mathfrak a_i^{\mu})_{\bar Q}$ is given by
	\begin{equation*}
		\gr_{k}(M_{30}^*,\mathfrak a_i^{\mu})_{\bar Q} = \begin{cases}
			-\mu - k(k+1)/2 - i + 1, & i \ge 0,\\
			-\mu - k(k+1)/2 - i , & i < 0,
		\end{cases}
	\end{equation*}
	for $\mu=0,1$.
\end{lem}
The following corollary concerning moduli spaces asymptotic to boundary-stable critical point is similar to \cite[Corollary~5.7]{KMOS2007} but slightly subtler. 
\begin{cor}	
\label{cor:M30_dim}
	If $\mathfrak a'$ is a boundary-stable critical point for which the corresponding $M_k(M_{30}^*,\mathfrak a')_{\bar Q}$ is nonempty, then one of the following must hold.
	\begin{enumerate}
		\item $k \in \{0,-1\}$ and $\mathfrak a' = \mathfrak a^0_0$, where $\gr_{k}(M_{30}^*,\mathfrak a')_{\bar Q} = 1$.
		\item $k \in \{1,-2\}$ and $\mathfrak a' = \mathfrak a^0_0$, where $\gr_{k}(M_{30}^*,\mathfrak a')_{\bar Q} = 0$.
		\item $k \in \{0,-1\}$ and $\mathfrak a' \in \{\mathfrak a^0_1,\mathfrak a^1_0\}$, where $\gr_{k}(M_{30}^*,\mathfrak a')_{\bar Q} = 0$.
	\end{enumerate}
	The reducible moduli space $M_k^{\red}(M_{30}^*,\mathfrak a')_{\bar Q}$ is empty whenever $\mathfrak a'$ is boundary-stable, except possibly in case (1) above.
	If $k \in \{0,-1\}$ and $\mathfrak a' = \mathfrak a^0_0$, then 
	\[
		\# M_k^{\red}(M_{30}^*,\mathfrak a^0_0)_{\bar Q} =0.
	\] 
	Therefore the element $\bar n_s$ is zero.
\end{cor}
\begin{proof}[Proof of Corollary~\ref{cor:M30_dim}]
	By Lemma~\ref{lem:M30_dim_stable}, the dimension of $M_k^{\red}(M_{30}^*,\mathfrak a')_{\bar Q}$ is negative as long as we are not in case (1).
	Assume $\mathfrak a' = \mathfrak a^0_0$ and $M_k^{\red}$ is zero-dimensional and nonempty.
	Then the compactification $M_k^+(M_{30}^*,\mathfrak a')_{\bar Q}$ is an 1-dimensional manifold with boundary, whose ends consist of the following.
	\begin{enumerate}[leftmargin=*,label=(\roman*)]
		\item $M_{z_0}(M_{30}^*,\mathfrak b)_{Q} \times M_{z_1}(\mathfrak b, \del M_{30}, \mathfrak a^0_0)$, where $\mathfrak b \in \{\mathfrak a^0_1,\mathfrak a^1_0\}$.
		\item  $M_{z_0}(M_{30}^*,\mathfrak b_1)_{Q} \times M_{z_1}(\mathfrak b_1, \del M_{30}, \mathfrak b_2) \times  M_{z_2}(\mathfrak b_2, \del M_{30}, \mathfrak a^0_0)$, \\
		where the middle term is boundary-obstructed.
		\item $M_k^{\red}(M_{30}^*,\mathfrak a^0_0)_{\bar Q}$
		\item $M_k(M_{30}^*,\mathfrak a^0_0)_{\pm \infty} \cong 
		M_{z_0}((\overline{\mathbb{CP}}^2 \setminus \Ball)^*, \mathfrak b)  \times M_{z_1}(\mathfrak b, (S^2 \times D^2 \setminus \Ball)^* ,\mathfrak a^0_0)$, \\
		where the  right hand side consists of irreducible real solutions on the disjoint union.
	\end{enumerate}
	Ends of type (i) and (ii) contribute zero mod two because of the trivial  differentials over $S_{30} = \Sigma_2(S^3,\text{U}_2)$.
	Type (iv) ends are empty, as the metrics at $\pm \infty$ were assumed to have positive scalar curvature on $\overline{\mathbb{CP}}^2$.
\end{proof} 
In the ordinary monopole Floer homology  \cite{KMOS2007},
for parity reason the critical points $\mathfrak a^0_j$ never appear in the zero dimensional moduli spaces.
In the real setting, both towers of reducibles over $S^1\times S^2$ contribute, as demonstrated by the following counterpart of \cite[Corollary~5.8]{KMOS2007}.
\begin{cor}
	There are two types of zero-dimensional moduli spaces that are asymptotic to boundary-unstable critical points:
	\begin{equation}
	\label{eqn:mod_space_Q}
			M_k(M_{30}^*,\mathfrak a^1_{i_k})_{\bar Q} \quad \text{and} \quad 	M_k(M_{30}^*,\mathfrak a^0_{i_k-1})_{\bar Q},
	\end{equation}
	where $i_k = -1-k(k+1)/2$. The moduli spaces consist entirely of reducibles.
\end{cor}

The element $\bar n_u \in C^u_{\bullet}(\del M_{30})$
 can written as a sum of two series
\begin{align*}
	\bar n_u 
	&= 
	\sum_k a_k^1 e^1_{i_k} +
	\sum_k a_k^0 e^0_{i_k -1} \\
	&=
	(a_0 + a_{-1})e^1_{-1} + 
	\sum_{j=2}^{\infty} (a_{-j}^1 + a^1_{-1+j} ) e^1_{-j}
	+ 
	\sum_{j=1}^{\infty} (a_{-j}^0 + a_{-1+j}^0 ) e^0_{-j},
\end{align*}
where $a_k^{\mu}$ counts the points in the moduli spaces in \eqref{eqn:mod_space_Q}. 
In the above expression, we have separated out the highest degree term and paired the summands belonging to conjugate spin\textsuperscript{c} structures.
We denote the map $\check L^{\mu}_{i}$ as the cobordism induced map where we only insert $e^{\mu}_i$.
Then $\check L$ can be written as the sum
\begin{equation*}
	\check L = (a_0+a_{-1})\check L^1_{-1} + 
	\sum_{j=2}^{\infty} (a_{-j}^1 + a^1_{-1+j} ) \check L^1_{-j}
	+ 
	\sum_{j=1}^{\infty} (a_{-j}^0 + a_{-1+j}^0 ) \check L^0_{-j},
\end{equation*}
where $(a_0+a_{-1})\check L^1_{-1}$ is the lowest order summand of $\check L$.

It suffices to show $(a_0+a_{-1})\check L^1_{-1}$ is an isomorphism, and this is a consequence of the following two lemmata.
\begin{lem}
\label{lem:sumis1}
	The sum $a_{k}^1 + a^1_{-1-k}$ is $1$ mod $2$ for any $k$.	
\end{lem}
\begin{lem}
\label{lem:L1-1}
	The map $\check L^1_{-1}$ is the identity map.
\end{lem}
\begin{proof}[Proof of Lemma~\ref{lem:sumis1}]
	The proof is the real version of \cite[Lemma~5.10]{KMOS2007}, based on the analysis of (perturbed) abelian anti-self-dual equations.
	Let $g$ be an invariant Riemannian metric on $N^*$, cylindrical along the ends.
	Let $\mathfrak t_k$ be a real spin\textsuperscript{c} structure.
	Since there is no first homology and no self-dual square-integrable harmonic $2$-forms on $N^*$, there is a unique spin\textsuperscript{c} connection $A(k,g)$ satisfying the abelian ASD equation with $L^2$-curvature.
	By uniqueness, $A(k,g)$ is invariant under the real structure.
	The real flat spin\textsuperscript{c} connection modulo gauge form a circle $\mathcal S$, just like the ordinary case. 
	So by looking at the limit of $A(k,g)$ (which is asymptotically flat), we obtain a map $\theta_k(g) \in \mathcal S$.
	By choosing a spin structure on $S_{30}$ and a compatible real structure, we obtain an involution on the circle $\sigma: \mathcal S \to \mathcal S$, with fixed point $\mathfrak s_+$ and $\mathfrak s_-$.
	The choice of real spin structure provides an identification between $\bar{\mathfrak t}_k$ and $\mathfrak t_{-1-k}$. 
	We have 
	\begin{equation*}
		\theta_{-1-k}(g) = \sigma\theta_k(g).
	\end{equation*}
	The map $\theta$ can be extended to broken metrics 
	\begin{equation*}
		\theta_k: \bar Q \to \mathcal S.
	\end{equation*}
	The extended map satisfies $\theta_k(\pm \infty) \in \{s_+,s_-\}$ and $\theta_{k}(+\infty) = \theta_{-1-k}(+\infty)$.
	We obtain a mod-2 cycle in $\mathcal S$ by gluing $\theta_{k},\theta_{-1-k}:[-\infty,+\infty] \to \mathcal S$.
	The key fact is the cycle $\theta_k \cup \theta_{1-k}$ is nonzero mod-2.
    This can be proved by considering an explicit model of the map $\theta_k$ as in \cite[Lemma~5.10]{KMOS2007}.
	
	Our lemma follows from the nonzero mod-2 degree.
	Indeed, there is a cobordism, for $x \in 
	\mathcal S$ generic,
	from
	\begin{equation*}
		\{x\} \times_{\mathcal S} (M^{ab}_k(M_{30}^*,\mathcal S) \cup M^{ab}_{-1-k}(M_{30}^*, \mathcal S)
	\end{equation*}
	whose count is the degree of the cycle, 
	to
	\begin{equation*}
		(M^{ab}_k(M_{20}, \mathfrak a^1)_{\bar Q} \cup M^{ab}_{-1-k}(M_{20}, \mathfrak a^1)_{\bar Q} ) \times 
		M^{ab}(\mathfrak a^1, ([0,1] \times (\mathbb S_{30} \cap M_{30})^*,\mathcal S) \times_{\mathcal S} \{x\}.
	\end{equation*}
	The superscript ``ab'' denotes moduli spaces of perturbed abelian anti-self-dual equation.
	The middle term $M^{ab}(\mathfrak a^1, ([0,1] \times (\mathbb S_{30} \cap M_{30})^*,\mathcal S) $ consists of solutions to the perturbed anti-self-dual equation where we take no perturbation at $t=\infty$, and perturbation at $t=-\infty$. 
	The fibre product  is obtained by taking the limit as $t \to \infty$.
	Since $M_k(M_{30}^*,\mathfrak a^1_{i_k})_{\bar Q} $ contains only reducibles, we have an identification of $M^{ab}_k(M_{20}, \mathfrak a^1)_{\bar Q} $ with $M_k(M_{30}^*,\mathfrak a^1_{i_k})_{\bar Q} $. 
\end{proof} 
\begin{proof}[Proof of Lemma~\ref{lem:L1-1}]
	Recall our previous observation: 
\begin{equation*}
	([0,3] \times \mathbb Y, \Sigma_{30}) = ([0,3] \times \mathbb Y - \mathbb B_{30}, [0,3] \times K_0 - \Delta) \cup_{(\mathbb S_{30}, \del  M_{30} )} (\mathbb B_{30}, M_{30}).
\end{equation*}
That is, if one cuts off $(\mathbb B_{30}, M_{30})$ from $([0,3] \times \mathbb Y,\Sigma_{30})$, then the resulting orbifold is the product cobordism with two standard discs in a 4-ball removed.
On the level of the double branched cover, if we remove $N$ from $W_{30}$ and glue back in an $S^1 \times B^3$, then we obtain the cylinder $[0,3] \times Y$.
The essence of the proof is to compare $W_{30}$ with the identity cobordism.	
Choose an invariant Riemannian metric on the branched cover $S^1 \times B^3$ of $B^4$ along $\Delta$ with positive scalar curvature. 
Assume that near the boundary $S_{30}$ the metric is cylindrical and standard.
We consider an 1-parameter family of metrics by inserting cylinder $[-T,T] \times S_{30}$ and perturbations supported on the inserted cylinders.
This provides us with a chain homotopy from the map
\begin{equation*}
	\check m([0,3] \times K_3):\check C_{\bullet}(K_3) \to C_{\bullet}(K_0)
\end{equation*}
to the map
\begin{equation*}
	\sum_{j < 0} b_j^{0} \check L^{0}_j  + \sum_{j < 0} b_j^{1} \check L^{1}_j
\end{equation*}
where
\begin{equation*}
	\bar n_u(\Delta)= \sum_{j < 0} b_j^{0} \check e^{0}_j  + \sum_{j < 0} b_j^{1} \check e^{1}_j.
\end{equation*}
and
\begin{equation*}
	b^{\mu}_j = \# M(\Delta, \mathfrak a^{\mu}_j).
\end{equation*}
But only when $j=-1$ and $\mu = 1$ the moduli space is zero-dimensional.
In particular, the count $b^1_{-1}$ is $1$ coming from the unique reducible solution to the perturbed Seiberg-Witten equations.
\end{proof}

\section{Examples}
\label{sec:exmp}
\begin{figure}[hbt]
	\includegraphics[width=3in]{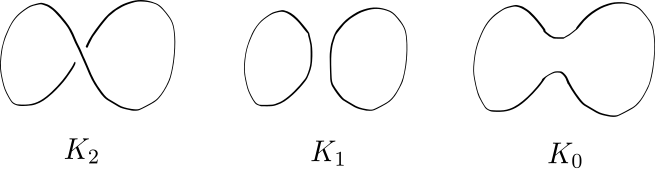}
	\caption{$\textbf{U}_1 \to \textbf{U}_2 \to \textbf{U}_1$}
\end{figure}
\subsection{$\textbf{U}_1 \to \textbf{U}_2 \to \textbf{U}_1$}
\hfill \break
Let $K_2 = K_0 = \text{U}_1$ and $K_1 = \text{U}_1$.
Let us describe the exact triangle for the ``bar'' version.
Let $\mathfrak s_0$ be unique torsion real spin\textsuperscript{c} structures on $\text{U}_2$.
Its real monopole Floer homology is
\[
	\overline{\HMR}_{\bullet}(\text{U}_2)
	= 
	\mathbb{F}_2[\upsilon^{-1},\upsilon]] \oplus \mathbb{F}_2[\upsilon^{-1},\upsilon]]\langle -1\rangle.
\]
The second tower is generated by critical point $\mathfrak a^0_0$ as an $\mathbb F[[\upsilon]]$-module, having absolute grading $(-1)$. 
The cobordism $\Sigma_{21}$ is branched covered by a punctured $S^2 \times D^2$, and only the spin structure is relevant to the cobordism map.
In particular,
\[
	\overline{\HMR}(\Sigma_{21}): \overline{\HMR}_{\bullet}(\text{U}_1)
	\to \overline{\HMR}_{\bullet}(\text{U}_2)
\]
takes $\overline{\HMR}_{\bullet}(\text{U}_1) \cong \mathbb F_2[\upsilon^{-1},\upsilon]]$ isomorphically to the tower $\mathbb{F}_2[\upsilon^{-1},\upsilon]]\langle -1\rangle$, as the  unique reducible non-blown-up Seiberg-Witten solution on $W_{21}$ restricts to the flat spin\textsuperscript{c} connection on $S^1 \times S^2$ at the minimum of the Morse function.
On the other hand, the map
\[
	\overline{\HMR}(\Sigma_{10}): \overline{\HMR}_{\bullet}(\text{U}_2)
	\to \overline{\HMR}_{\bullet}(\text{U}_1)
\]
maps the first $\mathbb{F}_2[\upsilon^{-1},\upsilon]]$-tower isomorphically to $\overline{\HMR}_{\bullet}(\text{U}_1)$ and the rest to zero.
It follows from the exact triangle that the map
\[
	\overline{\HMR}(\Sigma_{0,-1}): \overline{\HMR}_{\bullet}(\text{U}_1)
	\to \overline{\HMR}_{\bullet}(\text{U}_1)
\]
must be zero.
This can be seen directly by noticing the cobordism $W_{0,-1}$ is a twice punctured $\overline{\mathbb{CP}}^2$, and components for the conjugate spin\textsuperscript{c} structures on $W_{0,-1}$ cancel out modulo two.

\begin{figure}[hbt]
	\includegraphics[width=3in]{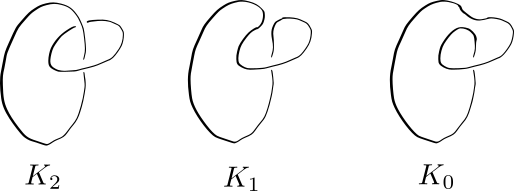}
	\caption{$\textbf{Hopf} \to \textbf{U}_1 \to \textbf{U}_1$}
\end{figure}
\subsection{$\textbf{Hopf} \to \textbf{U}_1 \to \textbf{U}_1$}
\hfill \break
Let $K_2$ be the Hopf link and $K_1 = K_0$ be the unknot.
The double branched cover of $K_2$ is $\mathbb{RP}^{3}$ which supports two self-conjugate spin\textsuperscript{c} structures $\mathfrak s_{\pm}$.
Cobordism 
$W_{21}$ is spin and negative definite.
Let $\mathfrak t_0$ be the spin structure on $W_{21}$ and suppose $\mathfrak s_+$ is the restriction of $\mathfrak t_0$.
Then $\overline{\HMR}(\Sigma_{21},\mathfrak t_0)$ is the leading order term of the series $\overline{\HMR}(\Sigma_{21})$ and is an isomorphism.
So $\overline{\HMR}(\Sigma_{21})$ maps $\overline{\HMR}_{
\bullet}(K_2,\mathfrak s_+)$ isomorphically to $\overline{\HMR}_{
\bullet}(K_0)$.
In fact, the conjugate symmetry $\mathfrak t \mapsto \bar{\mathfrak t}$ on $\overline{\HMR}(\Sigma_{21},\mathfrak t)$ and the fact that all spin\textsuperscript{c} structures involved are self-conjugate imply that $\overline{\HMR}(\Sigma_{21},\mathfrak t_0)$ is the only nonzero term while the rest cancel in pairs. 
The component $\overline{\HMR}_{
\bullet}(K_2,\mathfrak s_-)$ lies in the kernel of  $\overline{\HMR}(\Sigma_{21})$. 
Similarly, the cobordism $W_{32}$ is also spin and the cobordism map has a unique summand which maps $\overline{\HMR}_{
\bullet}(K_3)$  isomorphic to $\overline{\HMR}_{
\bullet}(K_2,\mathfrak s_+)$.

\begin{figure}[hbt]
	\includegraphics[width=3.2in]{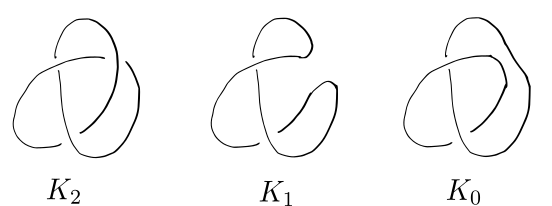}
	\caption{RHT $\to$ Unknot $\to$ Hopf link}
\end{figure}
\subsection{RHT $\to$ Unknot $\to$ Hopf link}
\hfill \break 
Let $K_2$ be the right-handed trefoil, $K_1$ be the unknot, and $K_0$ be the Hopf link.
They are minimal $L$-links, so it suffices to describe the $\overline{\HMR}$-version of the exact triangle.
Both $K_2$ and $K_0$ support multiple torsion spin\textsuperscript{c} structures.
Let $\mathfrak s_0$ be the self-conjugate real spin\textsuperscript{c} structure on $K_2$, and $\mathfrak s_1, \mathfrak s_2$ be the two other spin\textsuperscript{c} structures.
By the convention in Example~\ref{exmp:rational}, label the critical points on $K_2$ by $\{\mathfrak a_i^{\mu}\}$ for $i \in \mathbb Z$ and $\mu = 0,1,2$.
Similarly, label the critical points on $K_0$ by $\{\mathfrak a^{\pm}_j\}$, where $j \in \mathbb Z$, and critical points on $K_1$ by $\{\mathfrak a_j\}$.

The cobordism $W_{21}$ is obtained from 2-handle attachment and $H_2(W_{21},
\mathbb Z) \cong \mathbb Z$.
A generator of $H_2$ has self-intersection $(-3)$, so conjugation acts freely on the spin\textsuperscript{c} structures over $W_{21}$.
By the conjugate spin\textsuperscript{c} symmetry, for any spin\textsuperscript{c} structure $\mathfrak{t}$, 
\[
	\overline{\HMR}(\Sigma_{21}, \bar{\mathfrak t})(\mathfrak a^{1}_i) = \overline{\HMR}(\Sigma_{21}, \mathfrak t)(\mathfrak a^{2}_i),\quad 
	\overline{\HMR}(\Sigma_{21}, \bar{\mathfrak t})(\mathfrak a^{0}_i) = \overline{\HMR}(\Sigma_{21}, \mathfrak t)(\mathfrak a^{0}_i)
\]
from where we deduce that $\overline{\HMR}(\Sigma_{21})(\mathfrak a_j^0) = 0$.
Furthermore,  
\[ 
	\overline{\HMR}(\Sigma_{21}):\overline{\HMR}_{\bullet}(K_2) \to \overline{\HMR}_{\bullet}(K_1)
\]
	maps each of the $\{\mathfrak a^1_i\}, \{\mathfrak a^2_i\} $ towers isomorphically to the $\{\mathfrak a_i\}$ tower in $\overline{\HMR}_{\bullet}(K_1)$.
(This can be argued directly following the discussion in \cite[Section~4.14]{KMOS2007}.)
%Hence $\overline{\HMR}(\Sigma_{21})$ surjects onto $\overline{\HMR}_{\bullet}(K_1)$. 
The exact triangle implies $\overline{\HMR}(\Sigma_{10}) = 0$, which can alternatively be deduced from $b^+(W_{10}) > 0$.

The cobordism $W_{32}$ is spin. Assume $\mathfrak s_+$ is the unique spin\textsuperscript{c} structure that is the restriction of the unique spin structure $\mathfrak t_0$ on $W_{32}$, extending the unique spin structure on $K_2$.
Then the $\mathfrak t_0$ component of the cobordism map is an isomorphism onto the $\mathfrak s_0$-component of $\overline{\HMR}_{
\bullet}(K_2)$, and has the lowest absolute grading.
Moreover, the conjugation symmetry implies that the image of an element $\overline{\HMR}_{
\bullet}(K_3)$ under $\overline{\HMR}(\Sigma_{32})$ is symmetric in $\mathfrak s_1$ and $\mathfrak s_2$ components of $\overline{\HMR}_{
\bullet}(K_2)$. It follows that  $\overline{\HMR}(\Sigma_{32})$ injects into the kernel of $\overline{\HMR}(\Sigma_{21})$, as expected.
\section*{Acknowledgment}
I would like to thank my advisor Peter Kronheimer for his guidance and support.
I would also like to thank Hokuto Konno, Jin Miyazawa, and Masaki Taniguchi for several discussions and sharing their drafts.
\bibliographystyle{plain}
\bibliography{Triangle.bib}
\end{document}